\documentclass{amsart}
\pdfoutput=1
\usepackage{amsfonts,amscd,amssymb,amsmath,amsthm}
\usepackage{graphicx,mathrsfs,appendix}
\usepackage{xparse,enumerate}
\usepackage[T1]{fontenc}
\usepackage{hyperref}

\usepackage{tikz,pgfplots,pgfplotstable}
\usetikzlibrary{pgfplots.groupplots}
\usetikzlibrary{pgfplots.external}

\begin{document}

\newtheorem{theorem}{Theorem}[section]
\newtheorem{lemma}[theorem]{Lemma}
\newtheorem{corollary}[theorem]{Corollary}
\newtheorem{conjecture}[theorem]{Conjecture}
\newtheorem{cor}[theorem]{Corollary}
\newtheorem{proposition}[theorem]{Proposition}
\theoremstyle{definition}
\newtheorem{definition}[theorem]{Definition}
\newtheorem{example}[theorem]{Example}
\newtheorem{claim}[theorem]{Claim}
\newtheorem{remark}[theorem]{Remark}

\newenvironment{pfofthm}[1]
{\par\vskip2\parsep\noindent{\sc Proof of\ #1. }}{{\hfill
$\Box$}
\par\vskip2\parsep}
\newenvironment{pfoflem}[1]
{\par\vskip2\parsep\noindent{\sc Proof of Lemma\ #1. }}{{\hfill
$\Box$}
\par\vskip2\parsep}


\newcommand{\R}{\mathbb{R}}
\newcommand{\T}{\mathcal{T}}
\newcommand{\C}{\mathcal{C}}
\newcommand{\Z}{\mathbb{Z}}
\newcommand{\Q}{\mathbb{Q}}
\newcommand{\E}{\mathbb E}
\newcommand{\N}{\mathbb N}

\newcommand{\barray}{\begin{eqnarray*}}
\newcommand{\earray}{\end{eqnarray*}}

\newcommand{\beq}{\begin{equation}}
\newcommand{\eeq}{\end{equation}}

\def\ep{\varepsilon}


\renewcommand{\Pr}{\mathbb{P}}
\newcommand{\Prob}{\Pr}
\newcommand{\Var}{\operatorname{Var}}
\newcommand{\Cov}{\operatorname{Cov}}
\newcommand{\Exp}{\mathbb{E}}
\newcommand{\expect}{\mathbb{E}}
\newcommand{\1}{\mathbf{1}}
\newcommand{\prob}{\Pr}
\newcommand{\pr}{\Pr}
\newcommand{\filt}{\mathscr{F}}
\DeclareDocumentCommand \one { o }
{%
\IfNoValueTF {#1}
{\mathbf{1}  }
{\mathbf{1}\left\{ {#1} \right\} }%
}
\newcommand{\Bernoulli}{\operatorname{Bernoulli}}
\newcommand{\Binomial}{\operatorname{Binom}}
\newcommand{\Binom}{\Binomial}
\newcommand{\Poisson}{\operatorname{Poisson}}
\newcommand{\Exponential}{\operatorname{Exp}}
\newcommand{\Unif}{\operatorname{Unif}}

\newcommand{\cadlag}{\text{c\`adl\`ag}}
\newcommand{\weakto}{\ensuremath{\Rightarrow}}


\DeclareDocumentCommand \D { o }
{%
\IfNoValueTF {#1}
{D}
{\partial_{{#1}}}
}%
\newcommand{\sgn}{\operatorname{sgn}}

\newcommand{\dsp}[2]
{%
\tfrac{n + #1}{n+#2}x
}
\newcommand{\PDEOP}{\mathscr{S}^{\Psi}}
\newcommand{\PDEOPtilde}{{\widetilde{\mathscr{S}}}^{\Psi}}
\newcommand{\DSpace}{\mathcal{D}}
\newcommand{\Dspace}{\DSpace}
\newcommand{\Lloc}{L^1_{\text{loc}}}

\newcommand{\DTSpace}{\mathcal{X}}
\newcommand{\FPSpace}{\mathfrak{F}}
\newcommand{\Dto}{ \overset{\DTSpace}{\to} }
\newcommand{\aec}{\text{a.e.c.}}

\DeclareDocumentCommand \DTfn { m o O{}}
{
\IfNoValueTF {#2}
{#3 {\boldsymbol #1}}
{#3 {#1}_{#2}}
}

\DeclareDocumentCommand \A { o }{\DTfn{A}[#1]}
\DeclareDocumentCommand \Atilde { o }{\DTfn{A}[#1][\widetilde]}
\DeclareDocumentCommand \F { o }{\DTfn{F}[#1]}
\DeclareDocumentCommand \Fhat { o }{\DTfn{F}[#1][\widehat]}
\DeclareDocumentCommand \Ftilde { o }{\DTfn{F}[#1][\widetilde]}
\DeclareDocumentCommand \G { o }{\DTfn{G}[#1]}
\DeclareDocumentCommand \Ghat { o }{\DTfn{G}[#1][\widehat]}
\DeclareDocumentCommand \Gtilde { o }{\DTfn{G}[#1][\widetilde]}
\DeclareDocumentCommand \M { o }{\DTfn{M}[#1]}
\DeclareDocumentCommand \Mtilde { o }{\DTfn{M}[#1][\widetilde]}

\DeclareDocumentCommand \diF { O{n} }{D_{#1}}
\DeclareDocumentCommand \diFtilde { O{n} }{\widetilde D_{#1}}

%
%
%
%
%

\DeclareDocumentCommand \candy { O{\cdot} }
{%
\| {#1}
\|_{x^{-2}}
}%
\newcommand{\dcandy}{d_{x^{-2}}}
\newcommand{\dblah}{d_{\Lloc}}

\title{Choices and Intervals}
\author{Pascal Maillard}
\address{Department of Mathematics, Weizmann Institute of Science}
\email{pascal.maillard@weizmann.ac.il}
\author{Elliot Paquette}
\address{Department of Mathematics, Weizmann Institute of Science}
\email{elliot.paquette@gmail.com}
\thanks{
PM is supported by a grant from the Israel Science Foundation.
EP is supported by NSF Postdoctoral Fellowship DMS-1304057.
}
\date{\today}

\begin{abstract}
We consider a random interval splitting process, in which the splitting rule depends on the empirical distribution of interval lengths. We show that this empirical distribution converges to a limit almost surely as the number of intervals goes to infinity. We give a characterization of this limit as a solution of an ODE and use this to derive precise tail estimates. The convergence is established by showing that the size-biased empirical distribution evolves in the limit according to a certain deterministic evolution equation. Although this equation involves a non-local, non-linear operator, it can be studied thanks to a carefully chosen norm with respect to which this operator is contractive.

In finite-dimensional settings, convergence results like this usually go under the name of \emph{stochastic approximation} and can be approached by a general method of Kushner and Clark. An important technical contribution of this article is the extension of this method to an infinite-dimensional setting.
\end{abstract}

\maketitle

\section{Introduction}

Consider the following stochastic process on the unit circle.  At its initiation, finitely many distinct points are placed on the circle in any arbitrary configuration.  This configuration of points subdivides the circle into a finite number of intervals.  At each time step, two points are sampled uniformly from the circle.  Each of these points lands within some pair of intervals formed by the previous configuration.  Add the point that falls in the larger interval to the existing configuration of points, and discard the other.  If there is a tie, break it by flipping a fair coin, and continue adding points to the circle ad infinitum.  We call this process the \emph{max-$2$ process}.  If instead of keeping the points that fall in the larger intervals, we keep the points that fall in the smaller intervals, we call this process the \emph{min-$2$ process}.  If we simply choose between the two points uniformly at random, then we recover standard i.i.d. sampling of points from the circle, which we call the \emph{uniform process}.

Heuristically, the effect of having the two choices in the max-$2$ process should be to more evenly distribute the points around the circle than the uniform process.  In effect, the points repulse each other, as short intervals will be subdivided less frequently and large intervals will be subdivided more frequently.  In the min-$2$ process, on the other hand, points should have some tendency to clump together, so as to cause abnormally dense regions on the circle.  Nevertheless, we conjecture that in all cases, the limiting distribution of points is uniform on the circle (see paragraph ``Open problems'' below).

\subsection*{Main result}

In this article, we focus on the evolution of the law of a typical interval length. We first formalize the dynamics of the process. Let \(
I_1^{(n)},
I_2^{(n)},
\ldots,
I_{n+n_0}^{(n)}
\)
denote the lengths of the intervals after $n$ steps of the process (started with $n_0$ intervals). Define the size-biased empirical distribution function
\[
 \diFtilde[n](x) = \sum_{i=1}^{n+n_0} I_i^{(n)} \one[I_i^{(n)}\le x].
\]
This function is now defined to evolve according to Markovian dynamics as follows.
Given $\diFtilde[n]$, at the $(n+1)$-st step we choose an interval at random, with length $\ell_n = \diFtilde[n]^{-1}(u),$ where $u$ is sampled from a law on $(0,1]$ whose distribution function we denote by $\Psi$. This randomly chosen interval is now subdivided into two pieces at a point chosen uniformly inside the interval. This produces a new sequence of interval lengths
\(
I_1^{(n+1)},
I_2^{(n+1)},
\ldots,
I_{n+n_0+1}^{(n+1)}
\)
and the process is repeated. We call the resulting process the $\Psi$-process. Note that the max-$2$, uniform and min-$2$ processes are $\Psi$-processes with $\Psi(u) = u^2$, $u$ and $1-(1-u)^2$, respectively.

For $n\ge0$, denote by $\mu_n$ the empirical measure of the rescaled interval lengths $(n+n_0)I_1^{(n)},\ldots,(n+n_0)I_{n+n_0}^{(n)}$, i.e.\, the probability measure of sampling one of these lengths uniformly at random.  In symbols,
\[
\mu_n = \frac 1 {n+n_0} \sum_{i=1}^{n+n_0} \delta_{(n+n_0) I_i^{(n)}}.
\]
Set $\diF[n](x) = \diFtilde[n](x/(n+n_0))$ for $n\ge0$, $x\ge0$, so that $\diF[n](x) = \int_0^{x} y\,\mu_n(dy)$.
Our main theorem is the following:
\begin{theorem}
\label{thm:main0}
Assume that $\Psi$ is continuous and satisfies $1-\Psi(u) \ge c(1-u)^{\kappa_\Psi}$ for some $c>0$ and $\kappa_\Psi\in[1,\infty)$, for all $u\in(0,1)$. Then there is an absolutely continuous probability measure $\mu^\Psi$ on $(0,\infty)$ with mean 1, independent of the initial configuration, such that $\diF[n]$ converges pointwise to the function $F^{\Psi}(x) = \int_0^x y\,\mu^\Psi(dy)$, almost surely as $n\to\infty$. Furthermore, $\mu_n$ (weakly) converges to $\mu^\Psi$, almost surely as $n\to\infty$. The function $F^\Psi$ is the same as in Lemma~\ref{lem:unique_F*}.
\end{theorem}

A remark on the assumptions in Theorem~\ref{thm:main0}: we believe that continuity of $\Psi$ is not necessary for the theorem to hold. It is probably possible to extend our proof to cover the cases of discontinuous $\Psi$, at the expense of greater technicalities. However, we have not worked out the details. As for the second assumption, we first remark that a necessary condition for the theorem to hold is $\Psi(x) < 1$ for all $x < 1$. Under this condition, however, the entropy bounds obtained in Section~\ref{sec:entropy} would fail to hold, so some more restrictive estimates are fundamental for the current proof to work. The theorem might still be true with only the above condition, although the almost sure convergence might have to be replaced by convergence in probability.

Theorem~\ref{thm:main0} implies in particular that the max-$2$ process and the min-$2$ process have empirical interval distributions that converge, regardless the starting configuration, to a limit after rescaling (see Figure~\ref{fig:1}). This theorem also covers the analogous max-$k$ processes and min-$k$ processes for natural numbers $k,$ defined by first choosing $k$ points and then selecting the point in the largest or smallest interval respectively.  These are $\Psi$-processes with $\Psi(u) = u^k$ and $\Psi(u) = 1 - (1-u)^k$ respectively.

\begin{figure}[t]
\label{fig:1}
\begin{center}
\begin{tikzpicture}
\begin{axis}[
xlabel = {Length},
ylabel = {Density},
ymin = 0, ymax = 3,
xmin = 0, xmax = 4,
width = 10cm,
height = 7cm,
legend entries={min-$5$, min-$2$, uniform, max-$2$, max-$10$}
]
\addplot[red,thick] table {data/min-5-1G-00-intervals-4-xmax-1024-bins.dat};
\addplot[orange,thick] table {data/min-2-1G-15-intervals-4-xmax-1024-bins.dat};
\addplot[green,thick] table {data/unif-1G-00-intervals-4-xmax-1024-bins.dat};
\addplot[blue,thick] table {data/max-2-1G-15-intervals-4-xmax-1024-bins.dat};
\addplot[violet,thick] table {data/max-10-1G-00-intervals-4-xmax-1024-bins.dat};
\end{axis}
\end{tikzpicture}
\caption{Empirical density of interval lengths in simulation of max-10, max-2, uniform, min-2 and min-5 processes with $10^9$~points. For the plot, the x-axis has been discretized into 1024 equally sized bins.}
\end{center}
\end{figure}
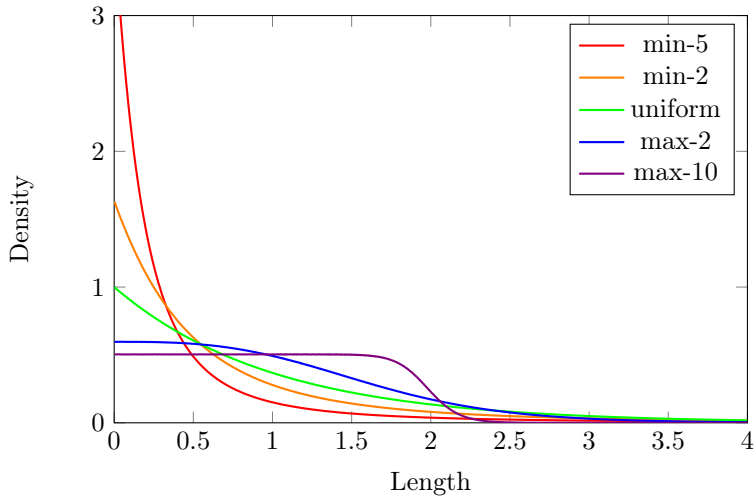

We also study properties of the limiting distribution $F=F^\Psi$. It is shown to be characterized by the following integro-differential equation
\[
F'(x) = x\int_x^\infty \frac 1 z\,d\Psi(F(z)),
\]
which allows us to derive tail estimates for many choices of $\Psi.$ Note that $f^{\Psi}(x) = F'(x)/x$ is the density of the (non-size-biased) empirical distribution. For the max-$k$ process, it is shown that $f^\Psi(x) \sim C_ke^{-kx}$ as $x\to\infty$ for some (implicit) $C_k$, while for the min-$k$ process the tail satisfies $f^\Psi(x)\sim (c_k/(k-1))x^{-2-1/(k-1)}$ for some explicit $c_k$ which satisfies $c_k\to 1$ as $k\to\infty$.  See Propositions~\ref{prop:F_tail_max} and \ref{prop:F_tail_min} for more precise statements. For comparison, in the uniform process, it is a classical theorem of~\cite{Weiss1955} that the limiting interval distribution is the exponential distribution of parameter $1$. Theorem~\ref{thm:main0} gives a new, complete proof of this fact. Many other precise results exist for the uniform splitting model, see for example \cite{Devroye1981,Devroye1982a,Deheuvels1982}.

Additionally, this theorem should be compared to results for the \emph{Kakutani interval splitting procedure} (see Lootgieter~\cite{Lootgieter1977}, van~Zwet~\cite{Zwet1978} and Slud~\cite{Slud1978} for results and further background on this process; note the correction~\cite{Slud1982} to the latter paper).  In its simplest form, this can be described by always taking $\ell_n$ to be the largest interval and then subdividing this interval by a uniformly chosen point.  Alternatively, it can be defined by letting $\Psi(u) = \one[u \geq 1]$ in the above definition (this case is not covered by Theorem~\ref{thm:main0}, but the proof could be adapted).  By a theorem of Pyke~\cite{Pyke1980}, the interval distribution of the Kakutani procedure converges to a $\Unif[0,2]$ variable.  Indeed, we can see that the max-$k$ process for large $k$ resembles the Kakutani process more and more, and in fact $F^{u^k}$ converges as $k \to \infty$ to the function $F^U(x)=x^2/4\wedge 1$, which is the size-biased distribution function of a $\Unif[0,2]$ variable (see Proposition~\ref{prop:limit_kakutani}).

\subsection*{Methodology}

We begin by embedding the discrete-time process $\diF[n](x)$ into a continuous time process $\A[t](x)$ in such a way that $n\approx e^t$. This continuous time process $\A[t]$ essentially evolves according to a stochastic evolution equation
\begin{equation}
\label{eq:SPDE}
\D[t] \A[t](x) = -x\D[x]\A[t](x) + x^2\int_x^\infty\frac{1}{y} d\Psi(\A[t](y)) + \DTfn{M}[t][\dot](x)
\end{equation}
for some centered noise $\M[t](x).$ This equation is both nonlinear and nonlocal, and thus it requires very specialized analysis.  First off, we transform the problem to studying an integrated form of the evolution, given by
\[
\A[t](x) = \A[0](e^{-t}x) + \int_0^t (e^{s-t}x)^2 \left[\int_{e^{s-t}x}^\infty \frac{1}{z} d\Psi(\A[s](z))\right]\,ds + \M[t](x).
\]
This allows to us to write
\(
\A = \PDEOP(\A) + \M,
\)
with $\PDEOP$ an operator acting on time-indexed distributions (here and throughout, we use boldface letters to denote function-valued processes indexed by time). Fixed points of $\PDEOP$ solve the following deterministic evolution equation:
\begin{equation}
\label{eq:deterministic_evo}
\F[t](x) = \F[0](e^{-t}x) + \int_0^t (e^{s-t}x)^2 \left[\int_{e^{s-t}x}^\infty \frac{1}{z} d\Psi(\F[s](z))\right]\,ds.
\end{equation}

Second, we show that \eqref{eq:deterministic_evo} has strong ergodicity properties. The key to this is the following carefully selected norm,
\[
\candy[f] = \int_0^\infty x^{-2}|f(x)|\,dx,
\]
with respect to which the evolution operator associated to \eqref{eq:deterministic_evo} quite surprisingly turns out to be a contraction (see Proposition~\ref{prop:candy}). This assures that there is a unique distribution $F^\Psi$ so that for any starting distribution, the large-time limit of the evolution is $F^\Psi$ (Lemma~\ref{lem:unique_F*}).

Third, we show how for any $\Psi$ satisfying the hypotheses of Theorem~\ref{thm:main0}, we can control the entropy of the size-biased empirical interval distribution. The aim of bounding the entropy is to establish tightness of the  family of distribution functions $\{\A[t]\}_{t\ge0}$.
One ingredient for this is an estimate for the size of the largest interval, which is shown to be smaller than $n^{-\alpha}$ for large $n$, for every $\alpha < (\kappa_\Psi+1)^{-1}$ (and under more restrictive conditions on $\Psi$, for every $\alpha < \kappa_\Psi^{-1}$).
We obtain these estimates by comparing the $\Psi$-process with the Kakutani process or the uniform process.

Finally, in order to show that $\A[t]$ converges to $F^\Psi$ despite the presence of noise, we adapt the Kushner--Clark method~\cite[Section 2.1]{Kushner1978}, which was developed for the study of stochastic approximation algorithms. To do so, we show that the sequence of shifted evolutions $\A[t]^{(n)} = \A[t+n]$ is almost surely precompact in a suitable topology, using the previously established tightness of the family $\{\A[t]\}_{t\ge0}$ together with an equicontinuity result. We then show that the limit points of this sequence are fixed points of the operator $\PDEOP$, from which we can conclude that the unique limit is the stationary evolution $\F^* \equiv F^\Psi$. This yields almost sure convergence of the stochastic evolution $\A$.

We remark that there exists a fairly extensive literature dealing with stochastic approximation in infinite-dimensional spaces (see e.g.\ \cite{Walk1977,Yin1992,Cardot2013} and the references therein). However, the results obtained there seem to be substantially too restrictive to apply to our setting. The most serious difficulty arises from the fact that the norm $\candy$, which is our only tool to study convergence of the (deterministic) evolution, is very sensitive to perturbations, due to the absolute value appearing inside the integral. As a consequence, we are not able to directly control the stochastic evolution $\A$ or the noise $\M$ in terms of this norm. For this reason, our proof of Theorem~\ref{thm:main0} does not yield any bounds on the rate of convergence of $\diF[n]$ to $F^\Psi$, although simulations indicate that this convergence is quite fast, possibly polynomial in $n$ (see Figure~\ref{fig:1}, in which the noise is completely invisible despite the high resolution of the data).

\pgfplotsset{
    every non boxed x axis/.style={}
}
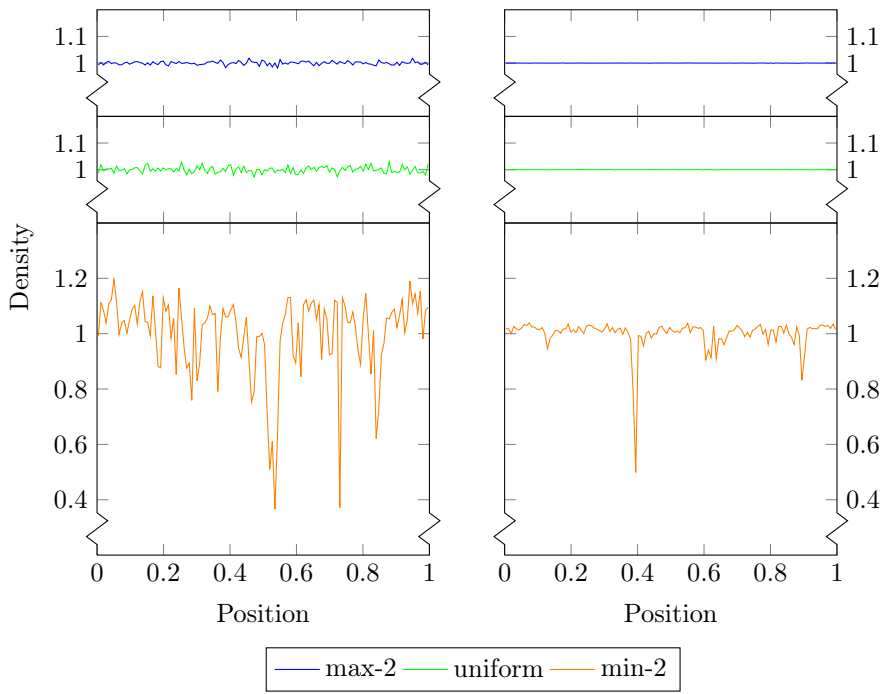
\begin{figure}[t]
\label{fig:2}
\begin{center}
\begin{tikzpicture}
\begin{groupplot}[
group style={
group size=2 by 3,
xticklabels at=edge bottom,
vertical sep=0pt,
xlabels at=edge bottom,
ylabels at=edge left
},
width=6cm,
xmin = 0, xmax = 1,
xlabel=Position,
]
\nextgroupplot[ymin=0.8,ymax=1.2,
	ytick={1,1.1},
	axis y discontinuity=crunch,
	legend to name=grouplegend,
	legend columns=-1,
	height=3.0cm]
\addplot[blue] table {./data/max-2-1M-00-positions-1-xmax-128-bins.dat};
\addlegendentry{max-$2$}
\addlegendimage{green}
\addlegendentry{uniform}
\addlegendimage{orange}
\addlegendentry{min-$2$}

\nextgroupplot[ymin=0.8,ymax=1.2,
	ytick={1,1.1},
	axis y discontinuity=crunch,
	yticklabel pos=right,
	height=3.0cm]
\addplot[blue] table {./data/max-2-1G-15-positions-1-xmax-128-bins.dat};

\nextgroupplot[ymin=0.8,ymax=1.2,
	ytick={1,1.1},
	axis y discontinuity=crunch,
	height=3.0cm]
\addplot[green] table {./data/unif-1M-00-positions-1-xmax-128-bins.dat};

\nextgroupplot[ymin=0.8,ymax=1.2,
	ytick={1,1.1},
	axis y discontinuity=crunch,
	yticklabel pos=right,
	height=3.0cm]
\addplot[green] table {./data/unif-1G-00-positions-1-xmax-128-bins.dat};

\nextgroupplot[ymin=0.2,ymax=1.4,
	ytick={0.4,0.6,0.8,1,1.2},
	axis y discontinuity=crunch,
	height=6.0cm]
\addplot[orange,samples=50] table {./data/min-2-1M-00-positions-1-xmax-128-bins.dat};

\nextgroupplot[ymin=0.2,ymax=1.4,
	ytick={0.4,0.6,0.8,1,1.2},
	axis y discontinuity=crunch,
	yticklabel pos=right,
	height=6.0cm]
\addplot[orange,samples=200] table {./data/min-2-1G-15-positions-1-xmax-128-bins.dat};

\end{groupplot}
\node at (5,-6.3) [inner sep=0pt,anchor=north, yshift=-5ex] {\ref{grouplegend}};
\node[label=left:\rotatebox{90}{Density}] at (-0.6,-2) {};

\end{tikzpicture}
\caption{Empirical density of points in the unit interval in simulation of max-2, uniform and min-2 processes with $10^6$ points (left) and $10^9$ points (right). For the plot, the x-axis has been discretized into 128 equally sized bins.}
\end{center}

\end{figure}

\subsection*{Discussion}

The max-$k$ choice and min-$k$ choice models are inspired by the general paradigm known as the ``power of 2 choices,'' which has seen considerable attention in the computer science and random graph literature~\cite{ABKU99,Achlioptas,riordan}.  Suppose one throws $n$ balls into $n$ bins, each uniformly at random, it is a simple exercise to see the maximum load (i.e. the number of balls in the fullest bin) is about $\log n/\log\log n.$ In Azar~et~al.\ \cite{ABKU99}, $n$ balls are thrown into $n$ bins, but for each ball, two bins are selected uniformly at random and the ball is placed in the bin with fewer balls.  This is seen to reduce the maximal number of balls in a bin to $\log_2 \log n,$ a considerable decrease from the same model without the two choices.  If one instead chooses the bin with the larger load, the maximal load increases to about $2\log n$ (see~\cite{DSKrM}). Similar considerations by the second author and Malyshkin~\cite{Malyshkin2013} show that the same conclusions hold in the min-choice case if the bins are sampled in a size-biased manner.

It is not clear to us whether there is a direct correspondence between the balls-and-bins model and our interval splitting process. However, in both models, the evolution of the large objects (the bins with high load/the large intervals) is simply accelerated by a factor of 2 in the max-version, whereas it is substantially slowed down in the min-version. To wit, in the uniform splitting model, the size of the largest interval is $\approx \log n/n$ \cite{Darling1953,Whitworth1897}. In the max-2 process, the tail of the interval distribution is of order $e^{-2x}$, which suggests that the size of the largest interval is $\approx \frac 1 2 \log n/n$. In the min-2 process on the other hand, the size of the largest interval is $n^{-1/2+o(1)}$ and thus on a completely different scale, mirroring what occurs in the balls-and-bins model (without size biasing).

There are many other interval subdivision models that are related directly or indirectly to the $\Psi$-process.  Brennan and Durrett~\cite{Brennan1987} study a model where each interval evolves independently, and an interval of length $L$ is subdivided with rate $L^{\alpha}.$  This is exactly the uniform process in the case $\alpha=1,$ and they show that the empirical interval distribution converges to a distribution with density proportional to $e^{-y^{\alpha}}.$  This work in turns sits within the larger class of fragmentation processes, see \cite{Bertoin2006} for a comprehensive account. Another, fairly different, interval split-merge model arises in the study of compositions of random transpositions, see \cite{Diaconis2004,Schramm2005}.

\subsection*{Open questions}

As mentioned above, Theorem~\ref{thm:main0} does not yield any information about the rate of convergence to the limiting interval distribution which therefore remains an open question. One could even expect a central limit theorem to hold.

The size of the largest interval in the process is a natural object to study. Here, we only have very crude estimates (see Section~\ref{sec:largest_interval}). One might expect that its magnitude can be deduced from the limiting interval distribution: it should be of the order of $\overline F^{-1}(1/n)/n$, where $\overline F$ is the tail of the (non-size-biased) limiting interval distribution.

Another interesting open problem is to study the spatial positions of the points in the $\Psi$-process. We believe that the limiting empirical distribution is always uniform (although the min-$k$ choice process displays extremely slow convergence, see Figure~\ref{fig:2}). This is indeed the case for the above-mentioned Kakutani process \cite{Lootgieter1977,Zwet1978,Slud1978}, but the methods do not carry over. One of the motivations for proving Theorem~\ref{thm:main0} is that it could help resolve that question. For a restricted class of $\Psi$-processes including the max-2-process, Matthew Junge \cite{Junge} has recently proved this conjecture by extending the methods from this article.

The problem of the spatial positions of the points originates with a problem posed to us by Itai Benjamini about a similar, albeit technically quite different problem.  Once again, consider throwing pairs of points on the circle.  Now, keep the point that is \emph{farthest} from other points and discard the point which is closest.  One can similarly define a process that does the reverse.  The evolution of the interval distribution in this case now becomes substantially more complicated, and simulations give very strong evidence that the limiting interval distributions are different.  Nevertheless, we expect that the points are almost surely equidistributed on the circle.  This problem can be naturally generalized to other classes of homogeneous spaces.

\subsection*{Overview of the article}

In Section~\ref{sec:definitions}, we introduce the main objects dealt with in this paper, among them a continuous version of the interval splitting process, the above-mentioned operator $\PDEOP$ and some functional spaces. Some fundamental properties of $\PDEOP$ are established in Section~\ref{sec:evo_prop}. Section~\ref{sec:candy_proof} proves the important Proposition~\ref{prop:candy}, which is the key to the existence of a unique limit $F^\Psi$ to the evolution equation. In Section~\ref{sec:largest_interval}, we turn to the stochastic evolution and give bounds for the size of the largest interval. In Section~\ref{sec:entropy}, we establish entropy bounds on the stochastic evolution used to yield tightness. Section~\ref{sec:convergence} then uses the results of the previous sections to prove convergence of the stochastic evolution $\A[t]$ to a deterministic limit. Section~\ref{sec:proof} contains the proof of Theorem~\ref{thm:main0} and of the portmanteau-type Lemma~\ref{lem:candy_portmanteau}. Finally, Section~\ref{sec:limit_profile} contains several results about properties of the limiting distribution.

\subsection*{Acknowledgements}

We are grateful to Itai Benjamini, who asked us a question which motivated this research.  We would also like to thank Matthew Junge and the referee for their close reading and helpful comments.  The computer simulations have been dutifully executed by the cluster of the Weizmann Institute of Science, Department of Mathematics and Computer Science. Node \texttt{n68} in particular has done a tremendous job and is hereby thanked.

\section{Definitions}
\label{sec:definitions}

In this section, we define the objects used in this article. All notation used in later sections is either defined there or in this section.

Throughout the paper, we will assume that $\Psi$ is the distribution function of a probability measure on $(0,1]$. Whenever we enforce stronger assumptions on $\Psi$, we will state them explicitly. The following two assumptions will appear quite often:
\begin{itemize}
	\item[(C)] $\Psi$ is continuous.
	\item[(D)] There exist $c>0$ and $\kappa_\Psi\in[1,\infty)$, such that $1-\Psi(u) \ge c(1-u)^{\kappa_\Psi}$ for all $u\in(0,1)$.
\end{itemize}

We define a continuous version of the $\Psi$-process which is technically convenient to work with.
Let $\Pi$ be a Poisson random measure on $[0,\infty) \times [0,1]^2$ with intensity
\(
e^tdt \otimes d\Psi(u) \otimes dv.
\)
We define a random family of distribution functions $(\Atilde[t])_{t\ge0}$ as follows: set $\ell_t(u) := \Atilde[t-]^{-1}(u)$ and define
\begin{align}
\label{eq:evolution}
\Atilde[t](x) &= \Atilde[0](x) + \sum_{(s,u,v)\in\Pi,\,s\le t} B(s,u,v,x),\quad \text{with} \\
\nonumber
B(s,u,v,x) &= \ell_s(u)\one[\ell_s(u) > x] \\
\nonumber
&\times \left(
v \one[\ell_s(u) v \leq x]
+(1-v)\one[\ell_s(u) (1-v) \leq x]
\right),
\end{align}
Note that by definition, $\Atilde[t](x)$ is increasing in $t$.

The relation between the process $(\Atilde[t])_{t\ge0}$ and the sequence $(\diFtilde[n])_{n \geq 0}$ defined in the introduction is the following: if we set $\Atilde[0] = \diFtilde[0],$ then with $\tau_n$ the time at which the $n$-th point appears in the Poisson process $\Pi$, we have $(\Atilde[\tau_n])_{n \geq 0}$ has the same distribution as $(\diFtilde[n])_{n \geq 0}.$

For every bounded Borel function $f$,
we have by definition
\[
\int_0^1 f(\ell_t(u))\,d\Psi(u) = \int_0^1 f(\Atilde[t-]^{-1}(u))\,d\Psi(u).
\]
Changing variables in the integral on the right yields the following useful formula:
\begin{equation}
\label{eq:f_ell}
\int_0^1 f(\ell_t(u))\,d\Psi(u) = \int_0^\infty f(z)\,d\Psi(\Atilde[t-](z)).
\end{equation}

Define the filtration $\filt = (\filt_t)_{t\ge0}$, where $\filt_t = \sigma(\Pi|_{[0,t]\times [0,1]^2})$. For every $x$, the process $(\Atilde[t](x))_{t\ge0}$ is a semimartingale with respect to $\filt$. In order to obtain its semimartingale decomposition, we need to calculate first and second moments of $B(t,u,v,x)$ conditioned on $\filt_{t-}$. By symmetry, we have for every $t$ and $x$,
\begin{align*}
\iint B(t,u,v,x)\,d\Psi(u)\,dv &= \int_0^1\ell_t(u)\one[\ell_t(u) > x] 2\int\limits_0^{ x /{\ell_t(u)}}v\,dv\,d\Psi(u)\\
&= x^2 \int_0^1 \frac 1 {\ell_t(u)}\one[\ell_t(u) > x]\,d\Psi(u),
\end{align*}
so that by \eqref{eq:f_ell},
\begin{equation}
\label{eq:B_exp}
\iint B(t,u,v,x)\,d\Psi(u)\,dv = x^2 \int_x^\infty \frac{1}{z}\,d\Psi(\Atilde[t-](z)).
\end{equation}
We therefore have for every $x\ge0$ the following semimartingale decomposition of $(\Atilde[t](x))_{t\ge0}$ (a detailed justification follows along the lines of the proof of Lemma~\ref{lem:Ht}):
\[
\Atilde[t](x) = \Atilde[0](x) + \int_0^t e^s x^2 \left[\int_x^\infty \frac{1}{z} d\Psi(\Atilde[s](z))\right]\,ds  + \Mtilde[t](x),
\]
for some local martingale $\Mtilde[t](x)$. Since $\int_x^\infty \frac{1}{z} d\Psi(\Atilde[s](z)) \leq 1/x$ and $|\Atilde[t](x)|\leq 1$ for all $t,x\geq 0$, $\Mtilde[t](x)$ is a martingale. Its quadratic variation will be calculated in Section~\ref{sec:convergence}.

We now define $\A[t](x) = \Atilde[t](e^{-t}x)$, so that
\[
\A[t](x) = \A[0](e^{-t}x) + \int_0^t (e^{s-t}x)^2 \left[\int_{e^{s-t}x}^\infty \frac{1}{z} d\Psi(\A[s](z))\right]\,ds + \M[t](x),
\]
with $\M[t](x) = \Mtilde[t](e^{-t}x)$.  We will see in Section~\ref{sec:convergence}, that as $t \to \infty,$ $\M[t]$ becomes vanishingly small in an appropriate norm on functions.  Thus, the function $\A[t]$ evolves to resemble a fixed point of a certain evolution, which we will now formalize.

Define the space $\Lloc$ of locally integrable functions $f:[0,\infty)\to\R$, endowed with the following canonical metric $\dblah$,
\[
\dblah(f,g) = \sum_{k=1}^\infty 2^{-k} \wedge \int_0^k |f(x)-g(x)|\,dx,
\]
which makes $\Lloc$ into a complete separable metric space. Define the subspace $\DSpace\subset\Lloc$ of subdistribution functions by
\[
\DSpace = \left\{
F : [0, \infty] \to [0,1],\text{ \cadlag, increasing}
\right\}.
\]
In this paper we will reserve the term \emph{distribution function} for cumulative distribution functions of probability measures. We include the following result for completeness:

\begin{lemma}
\label{lem:L1loc_pointwise}
For $F,F_1,F_2,\ldots\in\DSpace$, $F_n\to F$ with respect to $\dblah$ if and only if $F_n(x)\to F(x)$ at every point of continuity $x$ of $F$. 
\end{lemma}
\begin{proof}
The ``if'' direction follows from dominated convergence taking into account the fact that an increasing function has at most countably many discontinuities. The ``only if'' direction follows by elementary arguments: first, by considering  $\max(F_n,F)$ and $-\min(F_n,F)$ seperately, we can assume that $F_n\ge F$. If $x$ is a point of continuity of $F$ and $\ep>0$, let $\delta>0$ such that $F(x+\delta)-F(x) < \ep$. Then, by monotonicity of $F_n$ and $F$,
\[
 \int_x^{x+\delta} (F_n(y)-F(y))\,dy \ge \int_x^{x+\delta} (F_n(x)-F(x)-\ep)\,dy = \delta(F_n(x)-F(x)-\ep).
\]
Since the left-hand side vanishes as $n\to\infty$, we must have $F_n(x)-F(x) < 2\ep$ for large $n$. Since $\ep>0$ was arbitrary this implies $F_n(x)\to F(x)$ as $n\to\infty$ and finishes the proof.
\end{proof}

Lemma~\ref{lem:L1loc_pointwise} and Helly's selection theorem imply in particular that $(\DSpace,\dblah)$ is a compact metric space.

We will make extensive use of the following quantity. Define for $f\in\Lloc$:
\[
\candy[f] = \int_0^\infty x^{-2}|f(x)|\,dx\in[0,\infty].
\]
For $F \in \DSpace,$ we define the \emph{underlying measure} $x^{-1}dF.$  The next lemma identifies $\candy[F]$ as the mass of this underlying measure. 
\begin{lemma}
	Let $F \in \DSpace$ have $\candy[F] < \infty$ and let $dF$ be the Lebesgue--Stieltjes measure of $F,$ then $dF$ has no atom at $0$ and $\candy[F] = \int_0^\infty x^{-1}\,dF(x).$  
	\label{lem:candy_fubini}
\end{lemma}
\begin{proof}
	Let $F \in \DSpace$ have $\candy[F] <\infty$ and let $dF$ be as stated.  As for all $x \geq 0,$ $ F(x) \geq F(0),$ the finiteness of $\int_0^\infty x^{-2} F(x)\,dx$ implies that $F(0)=0,$ i.e.\,that $dF$ has no atom at $0.$  
Applying the Fubini-Tonelli theorem,
\begin{equation*}
	\candy[F]=
\int\limits_0^\infty x^{-2}\left[\int_0^x \,dF(z)\right]\,dx
=\int\limits_0^\infty \left[\int_z^\infty x^{-2}\,dx\right]\,dF(z)
=\int\limits_0^\infty x^{-1}\,dF(x).
\end{equation*}
\end{proof}

\noindent As a consequence, since $\diF[n](x) = \int_0^{x} y\,\mu_n(dy)$ and $\diFtilde[n](x) = \int_0^{x} y\,\widetilde{\mu}_n(dy)$ with $\widetilde{\mu}_n = \sum_{i=1}^{n+n_0} \delta_{I_i^{(n)}}$, we have that $\candy[ \diF[n] ] = \mu_n((0,\infty)) = 1$ and $\candy[\diFtilde[n]] = \widetilde{\mu}_n((0,\infty)) = n+n_0$ for all $n \geq 0.$  

Define a subspace $\DSpace_1\subset\DSpace$ by
\[
\DSpace_1 = \left\{
F\in\DSpace: \candy[F] \leq 1
\right\}.
\]
First note that by Fatou's lemma, $\DSpace_1$ is a closed subset of $(\DSpace,\dblah).$ We define a second metric on $\DSpace_1$ by $\dcandy(F,G) = \candy[F-G]$. Note that by definition, $\candy[F] \le 1$ for every $F\in\DSpace_1$, whence
\begin{equation}
\dcandy(F,G) \le \candy[F]+\candy[G]\le 2,\quad\forall F,G\in\DSpace_1.
\label{eq:dcandy_trivial_bound}
\end{equation}
Moreover, using $\dcandy,$ the space $\DSpace_1$ becomes a complete metric space.
\begin{lemma}
\label{lem:complete_candy}
The metric space $(\DSpace_1,\dcandy)$ is complete.
\end{lemma}
\begin{proof}
Note that for any $F,G \in \DSpace_1,$ we have that for any $K \in \N,$
\begin{equation}
\label{eq:dblah_domination}
\dblah(F,G) \leq 2^{-K} + \int_0^K |F(x) - G(x)|\,dx \leq 2^{-K} + K^2 \dcandy(F,G),
\end{equation}

	By \eqref{eq:dblah_domination}, it follows that if $(F_n)_{n \geq 0} \subset \DSpace_1$ is Cauchy in the $\dcandy$ metric, it is also Cauchy in the $\dblah$ metric.  In particular, given a $\dcandy$-Cauchy sequence $(F_n)_{n\ge0} \subset \DSpace_1,$ we have by the completeness of $\DSpace_1$ that there is an~$F\in\DSpace_1$ so that $\dblah(F_n,F) \to 0$ as $n\to\infty$.  
In particular, $F_m\to F$ pointwise almost everywhere as $m\to\infty.$ Fatou's lemma now yields for every $n\ge0$, $\dcandy(F_n,F) \le \liminf_{m\to\infty} \dcandy(F_n,F_m)$. Since $(F_n)_{n\ge0}$ is a  $\dcandy$-Cauchy sequence, this gives $\dcandy(F_n,F)\to0$ as $n\to\infty$, which finishes the proof.
\end{proof}

%

To prove Theorem \ref{thm:main0}, we will ultimately show that $\diF[n]$ converges in $\dblah$.  From this convergence, we would like to know in addition that $\mu_n$ converges.  
The following lemma gives a sufficient criterion to establish this convergence.  It further shows that this convergence is in fact equivalent to convergence of $\diF[n]$ in the $\dcandy$ sense.
\begin{lemma}
	Let $F,F_1,F_2,\ldots \in \DSpace$ be a sequence of subdistribution functions, all having finite $\candy$ norm.  The following are equivalent:
	\begin{enumerate}[(i)]
		\item $\dblah(F_n,F) \to 0$ and $\candy[F_n] \to \candy[F],$ 
		\item $\dcandy(F_n,F) \to 0,$ and
		\item $x^{-1}dF_n$ converges weakly to $x^{-1}dF.$
	\end{enumerate}
	\label{lem:candy_portmanteau}
\end{lemma}
\noindent We delay the proof to Section \ref{sec:proof}.




We now define the space  $\mathcal B([0,\infty),\Lloc)$ of Borel measurable maps from $[0,\infty)$ to $\Lloc$. Elements of this space will always be denoted by boldface characters when no index is present, such as $\F = (F_t)_{t\ge0}$. We endow this space with the topology of locally uniform convergence, which we denote by the symbol $\Dto$. Then $\F^{(n)} \Dto \F$ as $n\to\infty$ if and only if for all compact $K \subseteq [0,\infty)$ and all $t > 0,$
\[
\lim_{n \to \infty} \sup_{0 \leq s \leq t} \int\limits_K | \F[s]^{(n)}(x) - \F[s](x) |\,dx = 0.
\]
The subspaces $\DTSpace,\DTSpace_1\subset\mathcal B([0,\infty),\Lloc)$ are defined by
\[
\DTSpace = \mathcal B([0,\infty),\DSpace),\quad \DTSpace_1 = \mathcal B([0,\infty),\DSpace_1) \subset \DTSpace.
\]
Since $\DSpace$ and $\DSpace_1$ are closed subsets of $\Lloc$, $\DTSpace$ and $\DTSpace_1$ are closed subsets of $\mathcal B([0,\infty),\Lloc)$. Note that $\A,\Atilde\in\DTSpace$.

The spaces of continuous maps $C([0,\infty),\DSpace)$ and $C([0,\infty),\Lloc)$ are closed subsets of $\DTSpace$ and $B([0,\infty),\Lloc)$, respectively. Furthermore, the topology on these spaces can be metrized to make them complete separable metric spaces.

At last, between $\DTSpace$ and $C([0,\infty), \Lloc),$ we define the operator $\PDEOP$ given by
\[
\PDEOP(\F)_t(x) = {\F[0]}(e^{-t}x) + \int_0^t (e^{s-t}x)^2 \left[\int_{e^{s-t}x}^\infty \frac{1}{z} d\Psi({\F[s]}(z))\right]\,ds,
\]
and note that it allows us to write $\A = \PDEOP(\A) + \M.$ We will be interested in the following family of fixed points of the operator $\PDEOP$:
\[
\FPSpace = \{\F \in \DTSpace_1: \F = \PDEOP(\F), \forall t \geq 0: \F[t](+\infty) = 1 \text{ and } \{\F[t]\}_{t\geq 0} \text{ tight}.
\},
\]
Here, we recall that a family of distribution functions $\{\F[\alpha]\}_{\alpha\in X}$ on $[0,\infty)$ is tight if for all $\epsilon > 0,$ there is an $N > 0$ sufficiently large such that, for every $\alpha \in X,$ $\F[\alpha](N) > 1-\epsilon.$

\section{Properties of the operator \texorpdfstring{$\PDEOP$}{}}
\label{sec:evo_prop}

We will need to study $\FPSpace$ in the abstract, and one immediate concern is that $\FPSpace$ could be empty.  As a consequence of various compactness properties, we will use the stochastic evolution $\A$ to construct such fixed points for any continuous $\Psi$ in Section~\ref{sec:convergence}, although they could also be constructed through plain discretization.  For the moment, we will suppose that $\FPSpace$ is nonempty to establish some important properties of $\PDEOP$ and elements of $\FPSpace.$

Fixed points in $\FPSpace$ naturally admit a type of semigroup structure.  This in turn follows from the structure of $\PDEOP.$  To expose these properties, define the operator semigroup $T_t:\Lloc\to\Lloc$ by $T_t F(x) = F(e^{-t}x)$ and the operator $\mathcal{A}:\DSpace\to\Lloc$ by $\mathcal{A}F(x) = x^2\int_x^\infty z^{-1}\,d\Psi(F(z))$. Then, for every $t\ge0$,
\begin{equation}
\label{eq:PDEOP_TA}
\PDEOP(\F)_t = T_t\F[0] + \int_0^t T_{t-s}\mathcal A\F[s]\,ds.
\end{equation}
\begin{lemma}
\label{lem:composition}
Suppose $\F\in\DTSpace$. Define $\G = \F - \PDEOP(\F)$, such that $\G$ is a measurable function from $[0,\infty)$ to $\Lloc$. Define $\G[t]^{(s)} = \G[s+t]-T_t\G[s]$ and $\F[t]^{(s)} = \F[s+t]$ for every $s,t\ge0$. Then $\F^{(s)} = \PDEOP(\F^{(s)}) + \G^{(s)}$.
\end{lemma}
\begin{proof}
By \eqref{eq:PDEOP_TA}, we have for every $s,t\ge0$,
\begin{align*}
\PDEOP(\F^{(s)})_t
&= T_t\F[0]^{(s)} + \int_0^t T_{t-r}\mathcal A\F[r]^{(s)}\,dr\\
&= T_t\F[s] + \int_s^{s+t} T_{s+t-r}\mathcal{A}\F[r]\,dr\\
&= T_t(\PDEOP(\F)_s + \G[s]) + \int_s^{s+t} T_{s+t-r}\mathcal{A}\F[r]\,dr\\
&= T_t(T_s \F[0] + \int_0^s T_{s-r}\mathcal A\F[r]\,dr) + T_t\G[s] + \int_s^{s+t} T_{s+t-r}\mathcal{A}\F[r]\,dr\\
&= T_{s+t}\F[0] + \int_0^{s+t} T_{s+t-r}\mathcal{A}\F[r]\,dr +  T_t\G[s]\\
&= \PDEOP(\F)_{t+s} + T_t\G[s].
\end{align*}
By definition of $\G$ and $\G^{(s)}$, this finally yields for every $t\ge0$,
\[
\PDEOP(\F^{(s)})_t + \G^{(s)}_t = \PDEOP(\F)_{s+t} + T_t\G[s] + \G[s+t] - T_t\G[s] = \F[s+t] = \F[t]^{(s)},
\]
which implies that $\F^{(s)} = \PDEOP(\F^{(s)}) + \G^{(s)}$.
\end{proof}
\begin{corollary}
\label{cor:composition}
$\F \in \FPSpace$ implies that $\F^{(s)}$ in $\FPSpace$ for every $s\ge0$, where  $\F[t]^{(s)} = \F[s+t]$ for every $t\ge0$.
\end{corollary}

One critical property of the operator $\PDEOP$ is that it is continuous with respect to the topologies defined in the previous section.
\begin{lemma}
\label{lem:PDE_continuity}
Assume (C). Then
\(\PDEOP : \DTSpace \to C([0,\infty), \Lloc) \) is continuous.
\end{lemma}
\begin{proof}
Let $\F,\G\in\DTSpace.$  Fix a large $K >0.$ Then we have
\begin{align*}
\int\limits_0^K \left| \PDEOP(\F)_t(x) - \PDEOP(\G)_t(x) \right|\,dx\hspace{-1in}&\hspace{1in} \\
&\leq
\int\limits_0^K \left| \F[0](e^{-t}x) - \G[0](e^{-t}x) \right|\,dx \\
&+\int\limits_0^K \int\limits_0^t (e^{s-t}x)^2 \left|
\int_{e^{s-t}x}^\infty \frac{1}{z} (d\Psi(\F[s](z)) -  d\Psi(\G[s](z)))
\right|\,dsdx.
\end{align*}
The first of these two integrals clearly goes to $0$ uniformly on compact sets of $t$ if $\F \Dto \G$, since for every $t\ge0$, we have
\[
\int\limits_0^K \left| \F[0](e^{-t}x) - \G[0](e^{-t}x) \right|\,dx \le e^t \int\limits_0^{K} \left| \F[0](x) - \G[0](x) \right|\,dx,
\]
and the integral on the right-hand side goes to 0 by the definition of the convergence  $\F \Dto \G$.
We expand the interior of the second integral by parts so that for every $x$ which is a point of continuity of $\F[s]$ and $\G[s]$,
\begin{align*}
\left|\int_{e^{s-t}x}^\infty \frac{1}{z} (d\Psi(\F[s](z)) -  d\Psi(\G[s](z))) \right|
\hspace{-1in}&\hspace{1in} \\
&\leq
\frac{\left|
\Psi(\F[s](e^{s-t}x)) - \Psi(\G[s](e^{s-t}x))
\right|}{e^{s-t}x} \\
&+
\int_{e^{s-t}x}^\infty
\left|
\frac{\Psi(\F[s](z)) - \Psi(\G[s](z))}{z^2}
\right|\,dz.
\end{align*}
We handle the contribution of each of these pieces separately.  For the first integral, we would like to show that for any $T >0,$ if $\F \Dto \G,$
\begin{align*}
I_1 :=
\sup_{0 \leq t \leq T}
\int\limits_0^K \int\limits_0^t (e^{s-t}x)^2
\frac{\left|
\Psi(\F[s](e^{s-t}x)) - \Psi(\G[s](e^{s-t}x))
\right|}{e^{s-t}x}\,dsdx
\to 0.
\end{align*}
We reverse the order of integration, and change the integral to be over $y=e^{s-t}x,$ so that this is equivalent to
\begin{align*}
I_1 =
\sup_{0 \leq t \leq T}
\int\limits_0^t \int\limits_0^{e^{s-t}K}
y
\left|
\Psi(\F[s](y)) - \Psi(\G[s](y))
\right|\,dsdy.
\end{align*}
Noting that $s-t \leq 0,$ we may use the non-negativity of the integrand to conclude that
\[
I_1 \leq K \int\limits_0^T \int\limits_0^{K}
\left|
\Psi(\F[s](y)) - \Psi(\G[s](y))
\right|\,dsdy.
\]
For some bounded set $V$ of $\R^m$, recall that a sequence of measurable functions $f_n : V \to \R$ is said to converge to $0$ in measure if for all $\epsilon>0,$ $\lambda\{ x~:~|f_n| > \epsilon \} \to 0$ as $n \to \infty$ where $\lambda$ is Lebesgue measure.  For a sequence of bounded functions, convergence in measure is equivalent to $L^1$ convergence to $0$ on $V.$

Thus, under the assumption that $\F \Dto \G,$ for any bounded set of $(s,y)$ in $[0,\infty)^2,$ it follows that $\F[s](y) - \G[s](y)$ converges to $0$ in measure.  From the uniform continuity of $\Psi,$ it follows immediately that $\Psi(\F[s](y)) - \Psi(\G[s](y))$ converges to $0$ in measure.  Thus, from the boundedness of the integrand, we get that if $\F \Dto \G,$
\(
I_1 \to 0.
\)

We then truncate the second integral.  As before, we would like to show that
if $\F \Dto \G,$
\begin{align}
\label{eq:Bisou3}
I_2 :=
\sup_{0 \leq t \leq T}
\int\limits_0^K \int\limits_0^t (e^{s-t}x)^2
\int_{e^{s-t}x}^\infty
\left|
\frac{\Psi(\F[s](z)) - \Psi(\G[s](z))}{z^2}
\right|\,dz\,dsdx
\to 0.
\end{align}
Fix some $M > K,$ then we have that
\begin{align*}
\int_{e^{s-t}x}^\infty
\left|
\frac{\Psi(\F[s](z)) - \Psi(\G[s](z))}{z^2}
\right|\,dz.
\hspace{-1in}&\hspace{1in} \\
&
\leq\frac{2}{M}
+
\frac{1}{(e^{s-t}x)^2}
\int_{0}^M
\left|
{\Psi(\F[s](z)) - \Psi(\G[s](z))}\right|\,dz.
\end{align*}
Applying this bound to~\eqref{eq:Bisou3}, we get that there is a constant $C_{K,T}$ so that
\begin{align*}
I_2 \leq
\frac{C_{K,T}}{M}
+K \int\limits_0^T
\int\limits_{0}^M
\left|
{\Psi(\F[s](z)) - \Psi(\G[s](z))}
\right|\,dz\,ds.
\end{align*}
By the same argument used for $I_1$, if $\F \Dto \G,$ then for each fixed $M,$ this integral goes to $0.$  As we may then make $M$ as large as we like, we get that $I_2 \to 0$ as well.
\end{proof}

Our goal is ultimately to understand the large $t$ behavior of a function $\F \in \FPSpace$. This in essence requires us to show that the evolution operator associated to \eqref{eq:SPDE} has a type of ergodicity. This is achieved by the following proposition:
\begin{proposition}
\label{prop:candy}
For all $\F,\G\in\FPSpace$, for every $t\ge0$, we have
\[
\candy[ \F[t] - \G[t] ] \leq e^{-t} \candy[\F[0] - \G[0]].
\]
\end{proposition}
We delay the proof of this proposition until the next section.  The most central consequence of this proposition is that all $\F \in \FPSpace$ share a common large $t$ limit.
\begin{lemma}
\label{lem:unique_F*}
Assume $(C)$ and $\FPSpace \neq \emptyset$. 
Then
there is a unique distribution function $F^\Psi \in \DSpace_1$ so that for any $\F \in \FPSpace$
\[
\candy[\F[t] - F^\Psi] \leq 2e^{-t},\quad\forall t\ge0.
\]
Also, setting $\F^*\equiv F^\Psi$, then $\F^{*} \in \FPSpace.$
Furthermore, $F^\Psi$ is continuously differentiable and for all $x \geq 0,$
\begin{equation}
\label{eq:F_ode_integral}
(F^\Psi)'(x) = x\int_x^\infty \frac 1 z\,d\Psi(F^\Psi(z)).
\end{equation}
Finally, $\candy[F^\Psi] = 1$ and the Lebesgue--Stieltjes measure $x^{-1} dF^{\Psi}(x)$ is a  probability measure with mean 1. 
\end{lemma}
\begin{proof}
By assumption, there exists $\F\in\FPSpace$, i.e.\ a fixed point of the operator $\PDEOP$. We claim that the (transfinite) sequence $(\F[t])_{t\ge0}$ is Cauchy in $(\DSpace_1,\dcandy)$.
For this, let $s\le t$. By Corollary~\ref{cor:composition}, $\F^{(t-s)}\in\FPSpace$ as well. Proposition~\ref{prop:candy} and \eqref{eq:dcandy_trivial_bound} then imply,
\[
\candy[\F[s]-\F[t]] = \candy[\F[s] - \F[s]^{(t-s)}] \le e^{-s} \candy[\F[0] - \F[0]^{(t-s)}] \le 2e^{-s}.
\]
By Lemma~\ref{lem:complete_candy}, the space $\DSpace_1$ is complete under the metric $\dcandy$, which yields the existence of an $F \in \DSpace_1$ so that $\candy[\F[t] - F] \leq 2e^{-t}$ for every $t \geq 0.$

Suppose that $\G$ is another element of $\FPSpace.$  Then the same argument shows that there is a $G$ so that $\candy[\G[t] - G] \leq 2e^{-t}$ for every $t\ge0$. Therefore, for any $t \geq 0,$
\[
\dcandy(F, G) \leq
\dcandy(F, \G[t]) +
\dcandy(\F[t], \G[t]) +
\dcandy(\G[t], G) \leq 6e^{-t}.
\]
Hence, it follows that $F = G$ and so the limit function $F = F^\Psi$ is unique.  

On account of the tightness of the family $\{\F[t]\}_{t \geq 0},$ which holds by the definition of $\FPSpace,$ we have that $F^\Psi$ is a distribution function. 
As for the stationary evolution, we set $\F[t]^{*} = F^\Psi$ for all $t \geq 0.$  We need only check that this is indeed a fixed point.  Note that $\F^{(n)} \Dto \F^{*}$ as $n \to \infty$. By Assumption (C) and Lemma \ref{lem:PDE_continuity}, $\PDEOP$ is continuous and we may take limits on both sides of the equation
\(
\F^{(n)} = \PDEOP(\F^{(n)})
\)
to conclude that
\(
\F^{*} = \PDEOP(\F^{*}).
\)

Finally, we check the properties of  $F^\Psi$. We write from now on $F=F^\Psi$. As $\F^* = \PDEOP(F^*)$ and $\F[t]^*=F$ for all $t \geq 0,$
 we have that for each fixed $t\ge0,$ 
 \begin{equation}
\label{eq:F_stationary}
F(x) = F(e^{-t}x) + \int_0^t (e^{s-t}x)^2 \left[\int_{e^{s-t}x}^\infty \frac{1}{z} d\Psi(F(z))\right]\,ds.
\end{equation}
Let $u = e^{s-t}x,$ and change the outer integration to be over $u.$  Then we have
\[
  F(x) = F(e^{-t}x) + \int_{e^{-t}x}^x u \left[ \int_{u}^\infty \frac{1}{z} d\Psi(F(z))\right]\,du.
\]
As $\candy[F] \leq 1$, Lemma~\ref{lem:candy_fubini}
gives $F(0) = 0.$
Letting $t \to \infty,$ it follows from monotone convergence that for all $x\geq 0$
 \begin{equation}
\label{eq:F_stationary_2}
  F(x) = \int_{0}^x u \left[ \int_{u}^\infty \frac{1}{z} d\Psi(F(z))\right]\,du.
\end{equation}
As $g(u) = u\int_u^\infty z^{-1}d\Psi(F(z))$ is $L^1_{\operatorname{loc}}[0,\infty),$ it follows from \eqref{eq:F_stationary_2} that $F$ is continuous.  Since $F$ is continuous and $\Psi$ is continuous, it follows that $g$ is in fact continuous.  Hence by \eqref{eq:F_stationary_2}, we conclude that $F$ is continuously differentiable and \eqref{eq:F_ode_integral} holds.


From Lemma~\ref{lem:candy_fubini}, \eqref{eq:F_ode_integral} and an application of the Fubini-Tonelli theorem,
\begin{align*}
  \candy[F]
  =
  \int_0^\infty x^{-1}{F'(x)}\,dx = \int_0^\infty \left[\int_x^\infty \frac 1 z\,d\Psi(F(z))\right]\,dx
= \int_0^\infty d\Psi(F(z)) = 1.
\end{align*}
Hence $x^{-1}dF^{\Psi}(x)$ is a probability measure.  As $F^{\Psi}$ is a distribution function, $x^{-1}dF^{\Psi}(x)$ has mean $1.$ 
\end{proof}

\section{Proof of geometric decay of fixed points}
\label{sec:candy_proof}
In this section, we prove Proposition~\ref{prop:candy}. Let $\F,\G\in\FPSpace.$ We want to show that for every $t\ge0$,
\begin{equation}
\label{eq:candy_decay}
\candy[ \F[t] - \G[t] ] \leq e^{-t} \candy[\F[0] - \G[0]].
\end{equation}
We first define $\Ftilde[t](x) = \F[t](e^tx)$ and $\Gtilde[t](x) = \G[t](e^tx)$. Then $\Ftilde = \PDEOPtilde(\Ftilde)$ and  $\Gtilde = \PDEOPtilde(\Gtilde)$, where the operator $\PDEOPtilde$ is defined through
\[
\PDEOPtilde(\Ftilde)_t(x) = \Ftilde[0](x) + \int_0^t e^s x^2 \left[\int_x^\infty \frac{1}{z} d\Psi(\Ftilde[s](z))\right]\,ds.
\]
In particular, for every $x\ge0$, the map $t\mapsto \Ftilde[t](x)$
is absolutely continuous and its derivative is given by
\begin{equation}
\label{eq:Ftilde_derivative}
\D[t] 
\Ftilde[t](x)
= e^tx^2 \int_x^\infty \frac{1}{z} d\Psi(\Ftilde[t](z)).
\end{equation}
Since $\Ftilde[t]$ is $\cadlag$ for every $t$, the above formula then also holds jointly in $x$, for almost every $t\ge0$.

We now claim the following:
\begin{lemma}
\label{lem:candy_decay_tilde}
For every $t\ge 0$, we have
\[
\candy[ \Ftilde[t] - \Gtilde[t] ] \leq \candy[\Ftilde[0] - \Gtilde[0]].
\]
\end{lemma}
\begin{proof}
  Set $Y_t = \candy[ \Ftilde[t] - \Gtilde[t] ]$ and set $y_t(x) = | \Ftilde[t](x) - \Gtilde[t](x)|,$ so that $Y_t = \int_0^\infty x^{-2}y_t(x)\,dx.$  For each fixed $x \geq 0,$ the map $t \mapsto \Ftilde[t](x) - \Gtilde[t](x)$ is absolutely continuous.  
Hence, the map $t \mapsto y_t(x)$ is a composition of an absolutely continuous map with the Lipschitz map $u \mapsto |u|,$ from which it follows that $t \mapsto y_t(x)$ is absolutely continuous.  
Moreover, its derivative satisfies the following version of the chain rule for almost all $t \geq 0$ (see \cite[Theorem 3.44]{Leoni}):
\begin{equation}
	\D[t] y_t(x)
	= \sgn(\Ftilde[t](x)-\Gtilde[t](x)) 
	\D[t]( \Ftilde[t](x)-\Gtilde[t](x)),
	\label{eq:chain_rule}
\end{equation}
where the product is taken to be $0$ for all $t \geq 0$ for which $\D[t]( \Ftilde[t](x)-\Gtilde[t](x)) = 0.$

By \eqref{eq:Ftilde_derivative}, we have that for almost all $t \geq 0,$
\[
\D[t](\Ftilde[t](x)-\Gtilde[t](x)) =  e^t x^2 I_t(x),
\]
where we define for all $x,t \geq 0,$
\[
I_t(x) = \int_x^\infty \frac{\D[z](\Psi(\Ftilde[t](z))-\Psi(\Gtilde[t](z)))}{z}\,dz.
\]
Hence, we may take as a definition that for all $x,t\geq 0$ (with the convention $\sgn(0) = 0$, for example),
\[
\D[t]y_t(x) = e^{t}x^2
\sgn(\Ftilde[t](x)-\Gtilde[t](x))I_t(x).
\]
As for each $x \geq 0$ this definition satisfies \eqref{eq:chain_rule} for almost every $t\geq 0$ we have for \emph{all} $x,t \geq 0,$
\[
  y_t(x) = y_0(x) + \int_0^t \D[s]y_s(x)\,ds.
\]
In terms of $y_t(x),$ we may now write
\begin{equation}
  \label{eq:Yt_prefubini}
  Y_t 
  = \int_0^\infty x^{-2}y_t(x)\,dx
  = Y_0 + \int_0^\infty x^{-2} \int_0^t \D[s]y_s(x)\,dx
\end{equation}
By Tonelli's theorem, we have that for all $t \geq 0,$
\begin{align*}
  \int_0^\infty x^{-2} \int_0^t |\D[s]y_s(x)|\,dsdx
&\leq \int_0^\infty e^t \int_0^t |I_s(x)|\,dsdx \\
&\leq \int_0^\infty e^t \int_0^t \int_x^\infty z^{-1}(d\Psi(\Ftilde[s](z))+d\Psi(\Gtilde[s](z)))\,dzdsdx \\
&= e^t  \int_0^t \int_0^\infty \left[ \int_0^z 1\,dx\right] z^{-1}(d\Psi(\Ftilde[s](z))+d\Psi(\Gtilde[s](z)))\,dzds \\
&= e^t  \int_0^t \int_0^\infty (d\Psi(\Ftilde[s](z))+d\Psi(\Gtilde[s](z)))\,dzds \\
&= e^t  \int_0^t 2 ds = 2te^t.
\end{align*}
In particular, we may switch the order of integration in \eqref{eq:Yt_prefubini} to get
\begin{equation*}
  Y_t 
  = Y_0 + \int_0^t \int_0^\infty x^{-2}\D[s]y_s(x)\,dxds.
\end{equation*}
Therefore, if we can show that for almost every $t\geq 0$,
\begin{equation}
\label{eq:candy_derivative}
\int_0^\infty x^{-2} \D[t]y_t(x)\,dx \le 0,
\end{equation}
then $Y_t \le Y_0$ for all $t$, which proves the lemma.

We begin by applying integration by parts to $I_t(x),$ so that
\[
I_t(x)
= -x^{-1}(\Psi(\Ftilde[t](x)) - \Psi(\Gtilde[t](x))) + \int_x^\infty\frac{(\Psi(\Ftilde[t](z))-\Psi(\Gtilde[t](z)))}{z^2}\,dz.
\]
Recall that $\Psi$ is a non-decreasing function, so that $ \sgn(\Ftilde[t](x)-\Gtilde[t](x)) =  \sgn(\Psi(\Ftilde[t](x)) - \Psi(\Gtilde[t](x)))$ as long as $\Psi(\Ftilde[t](x)) \ne \Psi(\Gtilde[t](x))$. We may therefore bound
\begin{align*}
e^{-t}\D[t]y_t(x) 
&\leq  - x|\Psi(\Ftilde[t](x)) - \Psi(\Gtilde[t](x))| + x^2\int_x^\infty\frac{|\Psi(\Ftilde[t](z))-\Psi(\Gtilde[t](z))|}{z^2}\,dz.
\end{align*}
Multiply both sides by $x^{-2}$ and integrate in $x$ from $0$ to infinity:
\begin{multline*}
e^{-t}\int_0^\infty x^{-2}\D[t]y_t(x)\,dx \leq - \int_0^\infty\frac{|\Psi(\Ftilde[t](x)) - \Psi(\Gtilde[t](x))|}{x}\,dx \\
 + \int_0^\infty\int_x^\infty\frac{|\Psi(\Ftilde[t](z))-\Psi(\Gtilde[t](z))|}{z^2}\,dz\,dx.
\end{multline*}
The magic is that the last two integrals are actually equal. By the Fubini--Tonelli theorem,
\begin{align*}
\int_0^\infty\int_x^\infty\frac{|\Psi(\Ftilde[t](z))-\Psi(\Gtilde[t](z))|}{z^2}\,dz\,dx
&=\int_0^\infty\frac{|\Psi(\Ftilde[t](z))-\Psi(\Gtilde[t](z))|}{z^2}\left[\int_0^z1\,dx\right]\,dz \\
&=\int_0^\infty\frac{|\Psi(\Ftilde[t](z))-\Psi(\Gtilde[t](z))|}{z}\,dz.
\end{align*}
This implies \eqref{eq:candy_derivative} which concludes the proof of the lemma.
\end{proof}

\begin{proof}[Proof of Proposition~\ref{prop:candy}]
Recall that we have $\F[t](x) = \Ftilde[t](e^{-t}x)$ and $\G[t](x) = \Gtilde[t](e^{-t}x)$ for every $t\ge0$ and $x\ge0$. This gives for every $t\ge0$,
\begin{align*}
\candy[\F[t]-\G[t]] &= \int_0^\infty x^{-2} |\F[t](x)-\G[t](x)|\,dx\\
&= \int_0^\infty x^{-2} |\Ftilde[t](e^{-t}x)-\Gtilde[t](e^{-t}x)|\,dx\\
&=  e^{-t}\int_0^\infty x^{-2} |\Ftilde[t](x)-\Gtilde[t](x)|\,dx \\
&= e^{-t}\candy[\Ftilde[t]-\Gtilde[t]].
\end{align*}
The statement then follows from Lemma~\ref{lem:candy_decay_tilde}.
\end{proof}

\section{Bounds for the largest interval}
\label{sec:largest_interval}
Set $L_t= \max_{u\in(0,1)} \ell_t(u).$  We will begin by showing that $L_t \to 0$ at an exponential rate, using a comparison between the $\Psi$-process and the Kakutani process. We recall that assumption (D) is defined in Section~\ref{sec:definitions}.
\begin{lemma}
\label{lem:max_non_optimal}
Assume (D). Then, for every $\alpha \in(0,(\kappa_\Psi+1)^{-1})$, we have
\[
\Pr(\exists T: L_t \le e^{-\alpha t}\,\forall t\ge T) = 1.
\]
\end{lemma}
\begin{proof}
By the assumption on $\Psi$, the largest interval is split at rate at least $c e^t L_t^{\kappa_\Psi}$ at time $t$. $L_t$ is therefore dominated by the length $L'_t$ of the largest interval in the interval splitting process where \emph{only} the largest interval is split, and this at rate $c e^t (L'_t)^{\kappa_\Psi}$.  This process is a time changed version of the Kakutani process mentioned in the introduction. If $R_t$ denotes the number of times the largest interval has been split in this process, then it is known \cite{Lootgieter1977,Zwet1978} that provided $R_t\to\infty$ as $t\to\infty$,
\begin{equation}
\label{eq:1838}
 L'_tR_t \to 2\text{ almost surely, as $t\to\infty$.}
\end{equation}
Now fix $\alpha \in(0,(\kappa_\Psi+1)^{-1})$ and let $\delta > 0$ such that $\alpha\kappa_\Psi< 1-\alpha-\delta$. Let $t>0$. Since $L'_s$ is decreasing in $s$, we have for $t\ge 1$,
\[
 \int_0^t ce^s (L'_s)^{\kappa_\Psi}\,ds > c(L'_t)^{\kappa_\Psi}(e^t-1) > \frac c 2 (L'_t)^{\kappa_\Psi}e^t.
\]
Standard properties of Poisson processes then imply the existence of a Poisson distributed random variable $P$ with parameter $(c/2)e^{(\alpha+\delta) t}$, such that on the event $\{L'_t > e^{-\alpha t}\}$, we have $R_t \ge P$. In particular, Chebychev's inequality yields that for large $t$,
\[
 \Pr(L'_t > e^{-\alpha t},\,R_t < e^{(\alpha+\delta/2)t}) \le \Pr(P < e^{(\alpha+\delta/2)t}) < e^{-\alpha t}.
\]
It now follows from \eqref{eq:1838} and the Borel--Cantelli lemma applied to the previous equation that $L'_t \le e^{-\alpha t}$ for all large integers $t$. The lemma now follows (with any $\alpha' \in(0,\alpha)$ instead of $\alpha$) from the fact that $L'_t$ is decreasing in $t$. Since $\alpha \in(0,(\kappa_\Psi+1)^{-1})$ was arbitrary, this proves the lemma.
%
%
%
\end{proof}

The following lemma, which is not needed for the proof of Theorem~\ref{thm:main0}, gives the optimal exponent of the rate, under a more restrictive condition on $\Psi$. We believe that the result is true without this extra condition, but were not able to prove it.
\begin{lemma}
\label{lem:max}
Assume that $\Psi$ has an absolutely continuous component whose derivative $\psi$ satisfies $\psi(u)\ge c\kappa_\psi(1-u)^{\kappa_\psi-1}$ in a neighborhood of $1$, for some $\kappa_\psi \ge 1$. Then, for every $\alpha < \kappa_\Psi^{-1}$, we have
\[
\Pr(\exists T: L_t \le e^{-\alpha t}\,\forall t\ge T) = 1.
\]
\end{lemma}
\begin{proof}
Fix $\alpha < \kappa_\psi^{-1}$. Fix $t>0$. Let $R_t$ be the number of intervals of length greater than $e^{-\alpha t}$ at time $t$. We claim that there exists $\beta > 0$, such that $\Pr(R_t > 0) < e^{-\beta t}$ for large $t$. In order to show this, consider the evolution of the collection of intervals of length at least $e^{-\alpha t}$ between the times $0$ and $t$. By the definition of $\kappa_\psi$, if $t$ is sufficiently large, the rate at which an interval of length $\ell$ splits into two in this process is at least
\[
 ce^s\ell^{\kappa_\psi} \ge ce^se^{-(\kappa_\psi-1)\alpha t}\ell.
\]
This implies that
\(
\Pr(R_t > 0) \le \Pr(R'_t > 0),
\)
where $R'_t$ is the number of intervals of length greater than $e^{-\alpha t}$ at time $e^t-1$ in the process where an interval of length $\ell$ is split at rate $ce^{-(\kappa_\psi-1)\alpha t}\ell$, i.e.\ a time changed uniform process.

This corresponds to asking for the probability that the largest spacing is greater than $e^{-\alpha t}$ in a Poisson process on $[0,1]$ with intensity $c(e^{t}-1)e^{-(\kappa_\psi-1)\alpha t} \geq e^{(\alpha + \delta)t}$ for some positive $\delta$ and all $t$ sufficiently large.  Subdivide the interval into equally spaced intervals of length at most $e^{-\alpha t}/2$ and at least length $e^{-\alpha t}/3.$  Having a spacing larger than $e^{-\alpha t}$ implies one of these intervals has no points.  Applying a union bound, we get
\[
\Pr(R'_t > 0)
\leq
3e^{\alpha t}\exp(-e^{\delta t}/2)
\]

This shows that for some $\beta > 0$, for large $t$,
\[
\Pr(L_t > e^{-\alpha t}) = \Pr(R_t > 0) < e^{-\beta t}.
\]
The Borel-Cantelli lemma then implies that $L_n \le e^{-\alpha n}$ for large integers $n$ with probability one. Since $L_t$ is decreasing, this implies that almost surely, $L_t \le e^{\alpha(1-t)}$ for large $t$, which yields the lemma.
\end{proof}

\section{Entropy bounds}
\label{sec:entropy}

For a distribution function $F,$ define 
\[
H(F) = \int_0^\infty (\log x) dF(x),
\]
 if the integral exists.
 Let $\tilde H_t = H(\Atilde[t])$ and $H_t = H(\A[t])$.

 \begin{remark}
	 If we expand this definition for $F= \diFtilde,$ we get
	 \[
		 H(\diFtilde)
		 =
		 \sum_{i=1}^{n+n_0} I_i^{(n)} \log(I_i^{(n)}).
	 \]
	 This gives $-H(\diFtilde)$ the interpretation as the entropy of the discrete distribution $\diFtilde.$
 \end{remark}
 

\begin{lemma}
We have the following identities for the evolution of the entropy.  For all $t \geq 0,$
\label{lem:H_evo}
\begin{align}
\tilde H_t &= \tilde H_0 + \sum_{(s,u,v)\in \Pi,\,s\le t} {\ell_s(u)}W(v),\\
\label{eq:Ht_def}
H_t &= \tilde H_t + t = H_0 + \sum_{(s,u,v)\in \Pi,\,s\le t} {\ell_s(u)}W(v) + t,
\end{align}
where $W(v) = v \log v + (1-v) \log (1-v).$
\end{lemma}
This observation is also used by Lootgieter~\cite{Lootgieter1977} and in Slud~\cite{Slud1978}, and it follows from a simple calculation, which we include for completeness.
\begin{proof}
Note that the identity for $H_t$ follows immediately from the identity for $\tilde H_t$ on making the change of variables $y=e^{-t}x,$ and so we turn to the first identity.
From~\eqref{eq:evolution}, we have
\[
\tilde H_t = \tilde H_0
+
\sum_{(s,u,v)\in \Pi,\,s\le t}
\int_0^\infty (\log x)\,B(s,u,v,dx),
\]
where $B(s,u,v,dx)$ is the Lebesgue--Stieltjes measure associated to the function $x\mapsto B(s,u,v,x)$. It thus suffices to establish that for all $(s,u,v)\in[0,\infty)\times[0,1]^2$,
\[
\int_0^\infty (\log x)\,B(s,u,v,dx)
=
{\ell_s(u)}W(v).
\]
We now have that
\[
B(s,u,v,dx) = \ell_s(u)
v( \delta_{\ell_s(u) v}(x) - \delta_{\ell_s(u)}(x))
+\ell_s(u)(1-v)( \delta_{\ell_s(u)(1-v)}(x) - \delta_{\ell_s(u)}(x)),
\]
and hence
\begin{align*}
\int_0^\infty (\log x) B(s,u,v,dx)
\hspace{-1.0in}&
\hspace{1.0in}\\
&=
\ell_s(u)v( \log {\ell_s(u) v}- \log{\ell_s(u)})
+\ell_s(u)(1-v)(\log{\ell_s(u)(1-v)} - \log{\ell_s(u)})\\
&=\ell_s(u)W(v).
\end{align*}
This proves the lemma.
\end{proof}

Using Lemma~\ref{lem:H_evo}, we now calculate the drift and quadratic variation of $(H_t)_{t \geq 0}.$ Recall that for a semimartingale $X$, the \emph{predictable quadratic variation process} $\langle X\rangle$ is defined to be the compensator of the quadratic variation process $[X] = X^2 - 2\int X_- dX$ \cite[pp.66,122]{Protter}.
\begin{lemma}
\label{lem:Ht}
The process $H=(H_t)_{t\ge0}$ solves the following stochastic differential equation:
\[
dH_t = (1 - D_t)dt + dM_t,\quad D_t = \frac 1 2 \int_0^\infty z \,d\Psi(\A[t](z)),
\]
where $M=(M_t)_{t\ge0}$ is a martingale whose predictable quadratic variation satisfies
\[
d\langle M \rangle_t \le L_t^2\,dt.
\] 
Furthermore, we have for every $s<t$, $H_t - H_s \le t-s$.
\end{lemma}
\begin{proof}
  We begin by calculating the drift of $H.$  
  For any $t \geq 0,$ let $N_t$ be the number of points of $\Pi$ that have arrived by time $t,$ so that $N$ is a Poisson process with intensity $e^t.$
  Let $t_1 > t_0 \geq 0$ be any times.  Then 
  \[
	  \Exp \left[ \tilde H_{t_1} - \tilde H_{t_0} \vert \filt_{t_0} \vee \sigma( (N_t)_{t_0 \leq t \leq t_1}) \right]
	  =
	  \int_{t_0}^{t_1}
	  \int_0^1\int_0^1 \ell_{t}(u)W(v)\,d\Psi(u)\,dv\,dN_t.
  \]
By the change of variables formula \eqref{eq:f_ell}, we have for every $t\ge0$,
\begin{align*}
\int_0^1\int_0^1 \ell_t(u)W(v)\,d\Psi(u)dv
&=
\int_0^1\ell_t(u)\,d\Psi(u)
\int_0^1W(v)\,dv. \\
&= -\frac{1}{2}\int_0^1\ell_t(u)\,d\Psi(u) \\
&= -\frac{e^{-t}}{2}\int_0^\infty z \,d\Psi(\A[t-](z)).
\end{align*}
Hence, we have
\[
\Exp \left[ \tilde H_{t_1} - \tilde H_{t_0} \vert \filt_{t_0}\right]
=\Exp \left[
	-\frac{1}{2}\int_{t_0}^{t_1}e^{-t}
\int_0^\infty z \,d\Psi(\A[t-](z))\,dN_t
\,\biggl\vert\,\filt_{t_0}\right].
\]
From which it follows that $M',$ where 
\[
M'_{t} = \tilde H_t
	+\frac{1}{2}\int_{0}^{t}e^{-s}
\int_0^\infty z \,d\Psi(\A[s-](z))\,dN_s,
\]
is a martingale.

As $t \mapsto N_t - e^{t}$ is an $\filt$-adapted square-integrable martingale and $t \mapsto e^{-t}\int_0^\infty z \,d\Psi(\A[t-](z))$ is a non-negative predictable process bounded by 1 (again by \eqref{eq:f_ell} and the inequalities $0\leq \ell_t(u)\leq 1$), 
we conclude that its stochastic integral against $t \mapsto N_t - e^t$ has finite $\mathcal{H}^2$ norm (see \cite[pp. 155]{Protter}) and hence is a martingale.
Thus $M$, where
\[
M_{t} = \tilde H_t
+\frac{1}{2}\int_{0}^{t}
\int_0^\infty z \,d\Psi(\A[s-](z))\,ds,
\]
is a martingale.  As $\tilde H_t = H_t - t$ for all $t \geq 0$ and as the integral in the previous equation is indistinguishable from $\int_0^t D_s\,ds,$ we have shown that $dH_t = (1 - D_t)dt + dM_t.$


As for the quadratic variation, note that $[M] = [\tilde H]$ on account of their differing by a continuous process of finite variation.  Hence, we have the formula
\[
[M]_t
=
\sum_{(s,u,v)\in \Pi,\,s\le t} ({\ell_s(u)}W(v))^2,
\]
for all $t \geq 0.$  Identifying $\langle M\rangle$ can now be performed by the same sequence of steps performed to identify the drift of $\tilde H.$ 
In this case, we get
\[
  \Exp \left[ 
    \left[M\right]_{t_1}
    -\left[M\right]_{t_0}
    \,\vert\,\filt_{t_0}
\right]
  =
  \Exp \left[ 
    \int_{t_0}^{t_1}\int_0^1\int_0^1 \left(\ell_t(u)W(v)\right)^2\,d\Psi(u)dv\,dN_t
    \,\biggl\vert\,\filt_{t_0}
  \right].
\]
We now estimate
\begin{align*}
\int_0^1\int_0^1 \left(\ell_t(u)W(v)\right)^2\,d\Psi(u)dv
&=
\int_0^1(\ell_t(u))^2\,d\Psi(u)
\int_0^1W(v)^2\,dv. \\
&\leq
L_t^2 \int_0^1W(v)^2\,dv.
\end{align*}
We bound the integral of $W(v)^2$ by 
\[
\int_0^1W(v)^2\,dv
\leq 4\int_0^1 v^2 (\log v)^2\,dv
= \frac{8}{27}.
\]
From these estimates and a series of arguments similar to those made for $\tilde H,$ it follows that for all $t \geq 0,$
\[
d\langle M\rangle_t \leq 
\frac{8}{27}
L_t^2\,dt \le L_t^2\,dt.
\]
\end{proof}

The next lemma tells us that $D$ is large as soon as $H$ is large, and hence $H$ experiences a negative drift when $H$ grows too large.

\begin{lemma}
\label{lem:DF}
Assume (D). Then there exists a constant $C$, such that for any probability distribution function $F$,
\[
H(F) > C\text{ implies }D(F) := \frac 1 2 \int_0^\infty x\,d\Psi(F(x)) > 2.
\]
\end{lemma}
\begin{proof}
We first note that by integration by parts, we have
\begin{equation}
\label{eq:1313}
D(F) =
\frac 1 2 \int_0^\infty [1-\Psi(F(x))]\,dx,
\end{equation}
as well as
\begin{equation}
\label{eq:1314}
H(F) = -\int_0^1 \frac 1 x F(x)\,dx + \int_1^\infty \frac 1 x (1-F(x))\,dx \le \int_1^\infty \frac 1 x (1-F(x))\,dx.
\end{equation}
Now fix $x_0>1$. Note that for some $c'>0$, $1-\Psi(u) \ge 2c'(1-u)^{\kappa_\Psi}$ by assumption. Now set $\alpha := (2\kappa_\Psi)^{-1}$. If $1-F(x) > x^{-\alpha}$ for some $x> x_0$, then by \eqref{eq:1313} and the fact that $1-\Psi(F(x))$ is decreasing in $x$, $D(F) > c'x^{1/2} > c'x_0^{1/2}$. On the other hand,  if $1-F(x) \le x^{-\alpha}$ for all $x > x_0$, then  $H(F) \le x_0 + \alpha^{-1}$ by \eqref{eq:1314}. Hence, $H(F) > x_0+\alpha^{-1}$ implies $D(F) > c'x_0^{1/2}$. Choosing $x_0$ large enough and setting $C=x_0+\alpha^{-1}$ finishes the proof of the lemma.
\end{proof}

As a consequence of the negative drift and the decay of the quadratic variation, we have that $H$ is stochastically bounded for all time.

\begin{proposition}
\label{prop:entropy}
Assume (D). There exists a constant $C'$ such that
\[
\Pr(\exists T: H_t \le C'\,\forall t\ge T) = 1.
\]
\end{proposition}

\begin{proof}
Fix $\alpha \in(0,(\kappa_\Psi+1)^{-1})$ and $\beta>0$ and define $S_\beta$ to be the first time $t \geq \beta$, such that $L_t \ge e^{-\alpha t}$. Define the process  $H^\beta = (H^\beta_t)_{t\ge0}$ by
\(H_{t}^\beta = H_{(t+\beta) \wedge S_{\beta}}\).

We will show that there is a constant $C' >0$ so that for
every $\beta >0,$ with probability $1,$ $\limsup_{t \to \infty}~H^\beta_t~\leq C'.$  As a consequence the same statement holds with probability $1$ jointly for all $\beta \in \mathbb{N}.$  By Lemma~\ref{lem:max_non_optimal}, there is  with probability $1$ some $\beta^* \in \mathbb{N}$ so that $S_{\beta^*} = \infty,$ and hence we have that with probability $1,$ $\limsup_{t \to \infty } H_t \leq C',$ from which the proposition follows.

Let $C$ be the constant from Lemma~\ref{lem:DF}. We call an \emph{excursion} of the process $H^\beta$ a time interval $[t_1,t_2]$, such that $H^\beta_{t_1} \ge C+1$, and $t_2$ is the first time after $t_1$ that $H^\beta_{t_2} \le C$. We say that the excursion is \emph{successful}, if $H^\beta_t \ge C+2$ for some $t\in[t_1,t_2]$ and \emph{unsuccessful} otherwise. We further say that the process \emph{goes on an excursion} at the time $t$, if $t$ is the first time after the end of the last excursion that $H^\beta_t\ge C+1$.

Note that while the process $H^\beta$ is on an excursion, it has a drift $\le -1$ by Lemmas~\ref{lem:Ht} and \ref{lem:DF}. Furthermore, its jumps are bounded by 1 by definition \eqref{eq:Ht_def}. Standard calculations involving the optional stopping theorem now show that the excursion is finite almost surely. In order to prove the proposition, it is therefore enough to show that the number of successful excursions is finite almost surely. For this, denote by $T_1<T_2<\ldots$ the times at which the process goes on an excursion. By the last statement of Lemma~\ref{lem:Ht}, we have $T_{n+1} - T_n > 1$ for every $n$, whence $T_n > n$ for every $n$. Furthermore, denote
\[
P_n = \Pr(\text{the excursion starting at $T_n$ is successful}\,|\,\filt_{T_n}).
\]
By the Borel--Cantelli lemma, it is then enough to show that the sequence $P_n$ is summable almost surely.

For this, we first note that $H^\beta$ has no positive jumps, whence $H^\beta_{T_n} = C+1$ for every $n>1$. 
Fix $n \geq 1$ and let $\tau$ be the first time $t \geq 0$ that $H^\beta_{T_n + t} = C+2$ or $H^{\beta}_{T_n +t} \le C.$ 
Now define the process \( G_t = H^\beta_{T_n + (t\wedge \tau)}. \)
Note that it is possible that $S_\beta$ occurs strictly before $T_n + \tau,$ in which case $\tau = \infty$ and $G_t$ never reaches $C+2$ or $C.$

By Lemma~\ref{lem:Ht}, $G_t$ then satisfies by Lemma~\ref{lem:DF} that for $t< \tau$,
\[
  dG_t = (1 - D_{T_n+t})dt + dM_t \leq -dt + dM_t
\]
with a martingale $(M_t)_{t\ge0}.$ Its predictable quadratic variation satisfies for all $t \leq \tau,$
\[
  \langle M \rangle_t
  \leq \int_0^{t}
  L_{T_n +s}^2\,ds
  \leq \int_0^t e^{-2\alpha (T_n+s)}\,ds
  \leq \frac{1}{2\alpha} e^{-2\alpha n}.
\]
As $G_t$ is frozen for $t \geq \tau,$ we have in fact that $\langle M \rangle_t \leq \frac{1}{2\alpha} e^{-2\alpha n}$ for all $t \geq 0.$

Using that $G_t \leq M_t$ for all $t \geq 0,$ we have by Doob's $L^2$-martingale inequality that 
\[
\Pr( \sup_{t \geq 0} G_t \geq C+2)
\leq \Pr( \sup_{t \geq 0} M_t \geq 1)
\leq \sup_{t \geq 0} \Exp \langle M \rangle_t
\leq \frac{1}{2\alpha}e^{-2\alpha n}.
\]
This shows that $P_n < \frac{1}{2\alpha} e^{-2\alpha n}$ for every $n$. This sequence is summable, and the above arguments now permit us to conclude that the number of successful excursions is finite almost surely.

\end{proof}

\section{Convergence of the stochastic evolution}
\label{sec:convergence}
The goal of this section is to prove the following theorem.

\begin{theorem}
\label{thm:main}
Assume (C) and (D). Then $\FPSpace$ is nonempty. Furthermore, let $F^\Psi$ be the distribution function of Lemma~\ref{lem:unique_F*}. Then almost surely, as $t\to\infty$, $\A[t]\to F^\Psi$ pointwise.
\end{theorem}
As mentioned in the introduction, we will prove the theorem in a
manner that mirrors analogous methodology developed by Kushner and Clark \cite{Kushner1978} to handle the case of ODE. This relies heavily on compactness arguments for function spaces.

Say that a family $\{\F^{(n)}\}_{n \in \N}$ of functions in $\DTSpace$ is \emph{asymptotically equicontinuous} if for any compact $K \subset [0,\infty),$
\[
\lim_{\delta \to 0}
\lim_{n \to \infty}
\sup_{\substack{s,t\geq0 \\ |s-t|\leq \delta}}
\int\limits_K \left| \F[s]^{(n)}(x) - \F[t]^{(n)}(x) \right|\,dx
=0.
\]

To apply the argument we will establish the following properties of the stochastic evolution.
\begin{proposition}
\label{prop:evo_prop}
Assume (D). For the stochastic evolution $\A,$ the following hold almost surely:
\begin{enumerate}
\item The collection of distribution functions $\{\A[t]\}_{t \geq 0}$ is tight.
\item The family $\{\A^{(n)}\}_{n \geq 0}$ defined by $\A[t]^{(n)} = \A[t+n]$ for every $t\ge0$ is {asymptotically equicontinuous}.
\item The noise vanishes in the limit, i.e.\ $\M^{(n)} \Dto 0$ as $n \to \infty$, where $\M[t]^{(n)} = \M[t+n] - T_t \M[n]$ for every $t\ge0$.
\item Almost surely, $\int_0^\infty x^{-2}\A[t](x)\,dx \to 1$, as $t\to\infty$.
\end{enumerate}
\end{proposition}

Each of these claims are proven separately.  For convenience, we list where each piece is proven.  The tightness follows from the almost sure boundedness of entropy established by Proposition~\ref{prop:entropy}.  Asymptotic equicontinuity is proven in Lemma~\ref{lem:AEC}. The vanishing of the noise is proven in Lemma~\ref{lem:M_decay}. Finally, the convergence of the integrals follows from Lemma~\ref{lem:candy_bias}. We remark that assumption (D) is only used to establish the tightness claim.

Let us show how Proposition~\ref{prop:evo_prop} implies the Theorem~\ref{thm:main}.
We rely on the following consequence of Arzel\`{a}-Ascoli.
\begin{lemma}
\label{lem:AA}
Suppose that $\{\F^{(n)}\}_{n \in \N}$ is any family from $\DTSpace$ that is asymptotically equicontinuous, so that $\F^{(n)}$ is a step function when restricted to finite intervals,
and so that the entire collection
\(
\{ \F[t]^{(n)} \}_{t \in \R,n \in \N}
\)
is tight.  Then the family $\{ \F^{(n)} \}_{ n \in \N}$ is precompact and all its limit points $\F^{(\infty)}$ are in $C([0,\infty),\DSpace)$ and have that $\F[t]^{(\infty)}$ is a distribution function for each $t \geq 0.$
\end{lemma}
\begin{remark}
\label{rem:equicontinuity}
Since $\DTSpace$ is a metric space, precompactness in $\DTSpace$ is equivalent to existence of convergent subsequences. Also note that this lemma is still correct without the additional assumption that $\F^{(n)}$ be a step function.  We use this assumption simply to reduce the lemma to the standard Arzel\`{a}-Ascoli theorem.
\end{remark}
\begin{proof}
As $\F^{(n)}$ is a step function, we may define a piecewise linear interpolation $\Fhat^{(n)}.$  For any pair $(t_1,t_2)$ of consecutive jumps, we define
\[
\Fhat[t]^{(n)} = \frac{t-t_1}{t_2-t_1} \left[ \F[t_2]^{(n)} - \F[t_1]^{(n)} \right] + \F[t_1]^{(n)}
\]
on the interval $[t_1,t_2].$  Note this definition makes $\Fhat[t]^{(n)}$ equal to $\F[t]^{(n)}$ at jumps.

Further, from the convexity of the integral of the norm, we have for every compact $K\subset [0,\infty)$ and every $T\ge0,\delta >0$,
\[
\sup_{\substack{T \geq s,t\geq0 \\ |s-t|\leq \delta}}
\int\limits_K \left| \Fhat[s]^{(n)}(x) - \Fhat[t]^{(n)}(x) \right|\,dx
\leq \sup_{\substack{s,t\geq0 \\ |s-t|\leq \delta}}
\int\limits_K \left| \F[s]^{(n)}(x) - \F[t]^{(n)}(x) \right|\,dx,
\]
from which point the equicontinuity of the family $\{ \Fhat^{(n)} \}_{ n \in \N}$ is easily checked.
As $\DSpace$ is compact, we have by Arzel\`{a}-Ascoli that this sequence has convergent subsequences in $C([0,T],\DSpace)$ for each $T > 0.$  By diagonalization, we pick a convergent subsequence $n_k$ on $C([0,\infty),\DSpace)$ converging locally uniformly to some $\F^{(\infty)}.$  As for each $t \geq 0,$ this is the limit of a tight sequence of distribution functions $\F[t]^{(n_k)},$ it follows that $\F[t]^{(\infty)}$ is a distribution function for every $t\ge0$.
\end{proof}


\begin{proof}[Proof of Theorem~\ref{thm:main}]
Throughout the proof, all statements regarding the stochastic evolution $\A$ are meant to hold almost surely. From Lemma~\ref{lem:composition}, we have that
\begin{equation}
\label{eq:abs_master}
\A^{(n)} = \PDEOP(\A^{(n)}) + \M^{(n)}.
\end{equation}
By parts (1) and (2) of Proposition~\ref{prop:evo_prop}, and Lemma~\ref{lem:AA}, we may choose a sequence $\A^{(n_k)}$ which converges in $\DTSpace$ to an $\F^{(\infty)} \in C([0,\infty),\DSpace)$ with $\F[t]^{(\infty)}(+\infty) = 1$ for every $t\ge0$.
Taking limits in~\eqref{eq:abs_master}, we get
\[
\PDEOP( \A^{(n_k)} ) + \M^{(n_k)} \Dto \F^{(\infty)}.
\]
By part (3) of Proposition~\ref{prop:evo_prop}, we have that $\M^{(n_k)} \Dto 0,$ and by continuity of $\PDEOP$, we get that
\[
\PDEOP( \A^{(n_k)} ) \Dto \PDEOP( \F^{(\infty)} ).
\]
Thus $\F^{(\infty)}$ is a fixed point of $\PDEOP.$
By part (4) of Proposition~\ref{prop:evo_prop} and Fatou's lemma, we get $\candy[\F[t]^{(\infty)}] = 1$ for all $t \geq 0,$ whence $\F^{(\infty)} \in \FPSpace.$  In particular $\FPSpace$ is nonempty.

As $\FPSpace$ is nonempty, Lemma~\ref{lem:unique_F*} implies the existence of a unique distribution function $F^\Psi$ so that any evolution $\G \in \FPSpace$ has \(
\sup_{t \geq s} \candy[\G[t]-F^\Psi] \leq 2e^{-s}.
\)

We now turn to showing that 
$\A[t]\to F^\Psi$ 
with respect to $\dblah$
as $t\to\infty$ (this also easily implies that $\F^{(\infty)} = \F^* \equiv F^{\Psi}$ but we won't need this fact).
Let $\ep>0$. By \eqref{eq:dblah_domination},
there exists $\ep'>0$ such that $\dblah(G,H) \le \ep/2$ as soon as $\dcandy(G,H)<\ep'$. Now suppose there exists a sequence $(t_k)_{k\ge0}$ going to infinity such that $\dblah(\A[t_k],F^\Psi) > \ep$ for every $k$. Fix $T$, such that $2e^{-T} < \ep'$. As above, the sequence $(t_k - T)_{k\ge0}$ now contains a subsequence $(s_k)_{k\ge0}$, such that $\A^{(s_k)} \Dto \F^{(\aleph)}$ as $k\to\infty$, for some $\F^{(\aleph)}\in\FPSpace$. By Lemma~\ref{lem:unique_F*},
\[
 \candy[\F[T]^{(\aleph)}-F^\Psi] \leq 2e^{-T} < \ep',
\]
so that for large $k$, by the triangle inequality,
\[
\dblah(\A[T]^{(s_k)},F^\Psi) \le \dblah(\A[T]^{(s_k)},\F[T]^{(\aleph)}) + \dblah(\F[T]^{(\aleph)},F^\Psi) < \ep/2+\ep/2 = \ep.
\]
But since for every $k$, there exists $k'$ such that $\A[T]^{(s_k)} = \A[t_{k'}]$, this is in contradiction to the fact that  $\dblah(\A[t_k],F^\Psi) > \ep$ for every $k$.

Thus we have shown local $L^1$ convergence of the distribution function $\A[t]$ to $F^\Psi$ as $t\to \infty.$  From Lemma~\ref{lem:unique_F*}, $F^\Psi$ is continuous, and hence the convergence holds pointwise by Lemma~\ref{lem:L1loc_pointwise}.
%
\end{proof}

\subsection*{Asymptotic equicontinuity}

The next ingredient we need is the asymptotic equicontinuity of $\A^{(n)}.$
\begin{lemma}
\label{lem:AEC}
There is a $\delta_0 > 0$ and a constant $C$ so that
for every $0< \delta_1 < \delta_0$ there exists almost surely a $T_{\delta_1} <\infty$
so that
\[
\sup_{\substack{t \geq T_{\delta_1} \\ {0 \leq \delta \leq \delta_1} } }
\int_0^\infty \frac{|\A[t+\delta](x) - \A[t](x)|}{x^2}\,dx \leq C\delta_1.
\]
\end{lemma}
\noindent It follows that for any $\delta_1 < 0$ and any $M > 0,$ almost surely
\begin{align*}
\lim_{n \to \infty}
\sup_{\substack{s,t\geq0 \\ |s-t|\leq \delta}}
\int\limits_0^M \left| \A[s]^{(n)}(x) - \A[t]^{(n)}(x) \right|\,dx
&\leq 
\sup_{\substack{t \geq T_{\delta_1} \\ {0 \leq \delta \leq \delta_1} } }
\int_0^\infty \frac{M^2|\A[t+\delta](x) - \A[t](x)|}{x^2}\,dx  \\
&\leq M^2C\delta_1.
\end{align*}
As this holds jointly with probability $1$ for a countable sequence of $\delta_1$ going to $0$ and $M \in \mathbb{N},$ the almost sure asymptotic equicontinuity of $\left\{ A^{(n)} \right\}_{n \geq 0}$ follows.

This lemma depends very weakly on the details of the interval splitting procedure outlined in~\eqref{eq:evolution}.  The only randomness that needs to be considered are fluctuations in the times at which the points appear under the law of the Poisson process.
\begin{lemma}
\label{lem:AECsupport}
Recall that $N_t$ is the number of points of $\Pi$ with first coordinate in $[0,t].$
There is a $\delta_0 > 0$ so that
for every $0< \delta < \delta_0$ there exists almost surely a $T_{\delta} <\infty $ so that
\[
\sup_{t \geq T_{\delta}}
\left[N_{t+\delta} - N_t\right] \leq 2\delta e^t.
\]
\end{lemma}
\begin{proof}
Set $Q(t) = e^{t}-1,$ so that $N_{Q^{-1}(t)}$ is a standard Poisson process.  By the law of large numbers, \( N_{Q^{-1}(t)}/t \to 1 \) almost surely as $t \to \infty$.  Then
\begin{align*}
\limsup_{t \to \infty} \frac{N_{t+\delta} - N_t}{e^{t}}
&=\limsup_{t \to \infty} \frac{N_{t+\delta} - N_t}{Q(t)} \\
&=\limsup_{t \to \infty} \frac{N_{t+\delta}}{Q(t+\delta)}\frac{Q(t+\delta)}{Q(t)}  - \frac{N_t}{Q(t)} = e^{\delta} - 1
\end{align*}
almost surely.
Hence, choosing $\delta_0$ sufficiently small that $\delta < \delta_0$ implies $e^{\delta} -1 \leq 2\delta,$ the proof is complete.
\end{proof}
\begin{lemma}
  \label{lem:candy_bias} The following statements hold:
\begin{enumerate}
	\item $\int_0^\infty x^{-2}\Atilde[t](x)\,dx = n_0 + N_t$.
	\item $\int_0^\infty x^{-2}\A[t](x)\,dx\to1$ almost surely as $t\to\infty$.
\end{enumerate}
\end{lemma}
\begin{proof}
  The first observation is an immediate consequence of Lemma~\ref{lem:candy_fubini}.  The second observation follows from changing variables
\[
\int_0^\infty x^{-2}\A[t](x)\,dx
=e^{-t}\int_0^\infty x^{-2}\Atilde[t](x)\,dx = e^{-t}(n_0 +N_t).
\]
As $e^{-t}N_t \to 1$ almost surely, we have completed the proof.
\end{proof}
With this in hand, we now turn to proving Lemma~\ref{lem:AEC}.
\begin{proof}[Proof of Lemma~\ref{lem:AEC}]
We begin by changing variables to remove the spatial scaling of the distribution functions
\begin{align}
\label{eq:AEC1}
\int_0^\infty \frac{|\A[t+\delta](x) - \A[t](x)|}{x^2}\,dx
=\int_0^\infty e^{-t}\frac{|\Atilde[t+\delta](e^{-\delta}x) - \Atilde[t](x)|}{x^2}\,dx.
\end{align}
The key observation is a pair of domination relations that vastly simplify the integral.  On the one hand, from the fact that $\Atilde[t+\delta]$ is nondecreasing, we have that
\(
\Atilde[t+\delta](e^{-\delta}x) \leq \Atilde[t+\delta](x)
\)
for all $x \geq 0.$  On the other hand, from the fact that $\Atilde[t+\delta]$ was built from $\Atilde[t]$ by adding non-negative functions, we have that
\(
\Atilde[t+\delta](x) \geq \Atilde[t](x).
\)
Thus on applying both of these observations to~\eqref{eq:AEC1} we have that
\begin{align}
\label{eq:AEC2}
e^{-t}\int_0^\infty \frac{|\Atilde[t+\delta](e^{-\delta}x) - \Atilde[t](x)|}{x^2}\,dx
\leq
 ~&e^{-t}\int_0^\infty\frac{\Atilde[t+\delta](x) - \Atilde[t+\delta](e^{-\delta}x)}{x^2}\,dx \\
 \nonumber
+
 &e^{-t}\int_0^\infty\frac{\Atilde[t+\delta](x) - \Atilde[t](x)}{x^2}\,dx.
\end{align}
By applying Lemma~\ref{lem:candy_bias}, Lemma~\ref{lem:candy_fubini} and a change of variables, the first of these integrals can be calculated exactly:
\begin{align}
\label{eq:AEC3}
e^{-t}\int_0^\infty\frac{\Atilde[t+\delta](x) - \Atilde[t+\delta](e^{-\delta}x)}{x^2}\,dx = e^{-t}N_{t+\delta}(1-e^{-\delta}).
\end{align}
We can also calculate the second integral exactly
\begin{align}
\label{eq:AEC4}
 e^{-t}\int_0^\infty\frac{\Atilde[t+\delta](x) - \Atilde[t](x)}{x^2}\,dx
 =e^{-t}\left(
 N_{t+\delta} - N_t
 \right).
\end{align}
Combining \eqref{eq:AEC1}, \eqref{eq:AEC2}, \eqref{eq:AEC3} and \eqref{eq:AEC4} and using the monotonicity of \eqref{eq:AEC3} and \eqref{eq:AEC4} in $\delta,$ we get that
\begin{equation}
  \sup_{0 \leq \delta \leq \delta_1}  
\int_0^\infty \frac{|\A[t+\delta](x) - \A[t](x)|}{x^2}\,dx
\leq e^{-t}N_{t+\delta_1}(1 - e^{-\delta_1}) + e^{-t}(N_{t+\delta_1} - N_t).
  \label{eq:AEC5}
\end{equation}

By Lemma~\ref{lem:candy_bias}, there is a $\delta_0 > 0$ so that
for every $0< \delta_1 < \delta_0$ there exists almost surely a $T_{\delta_1} < \infty$ so that
\[
\sup_{t \geq T_{\delta_1}}
\left[N_{t+\delta_1} - N_t\right] \leq 2\delta_1 e^t.
\]
Similarly, we get that there is a $T < \infty $ so that $N_{t} \leq 2e^{t}$ for all $t > T.$
Hence, combining this with \eqref{eq:AEC5} we get that
\[
\sup_{t \geq T_\delta \wedge T}
  \sup_{0 \leq \delta \leq \delta_1}  
\int_0^\infty \frac{|\A[t+\delta](x) - \A[t](x)|}{x^2}\,dx \leq 2(e^{\delta_1}-1) + 2\delta_1,
\]
so that picking $C>0$ sufficiently large, we have completed the proof.
\end{proof}

\subsection*{Decay of the noise}

The remaining condition to check is that $\M^{(n)} \Dto 0.$  Consider the process $(I_t)_{t\ge0}$ defined by
\[
I_{t} = \int_0^\infty\frac{1}{x^3}( \M[t](x))^2\,dx.
\]
We will show that $I_t \to 0,$ but before doing so, let us see how this implies that $\M^{(n)} \Dto 0.$  Recall that $\M[t]^{(n)}= \M[t+n] - T_t\M[n].$
\begin{align*}
\int\limits_0^\infty \frac 1 {x^3} \left(\M[t]^{(n)}(x)\right)^2\,dx
&=\int\limits_0^\infty \frac{1}{x^3} \left( \M[t+n](x) - \M[n](e^{-t}x)\right)^2\,dx \\
&\leq 2\int\limits_0^\infty \frac{1}{x^3} \left[\left( \M[t+n](x) \right)^2
+ \left(\M[n](e^{-t}x)\right)^2\right]\,dx \\
&\leq 2I_{t+n} +
e^{-2t}\int\limits_0^\infty \frac{1}{x^3}( \M[n](x))^2\,dx,\\
\intertext{where we have made a change of variables for the last inequality.  Thus we conclude}
\int\limits_0^\infty \frac 1 {x^3} \left(\M[t]^{(n)}(x)\right)^2\,dx
&\leq 2I_{t+n} + e^{-2t}I_{n}.
\end{align*}
In particular if $I_t \to 0,$ then
\[
\sup_{t \geq 0}\int_0^\infty \frac{1}{x^3} \left( \M[t]^{(n)}(x) \right)^2\,dx \to 0
\]
as $n \to \infty.$  Then, for any fixed $k > 0$ and any $t \geq 0,$ we may apply Cauchy-Schwarz to conclude that
\begin{align*}
\int_0^k \left| \M[t]^{(n)}(x) \right|\,dx
&=\int_0^k x^{3/2}x^{-3/2}\left| \M[t]^{(n)}(x) \right|\,dx \\
&\leq \frac{k^2}{2} \biggl[\int_0^k x^{-3}\left| \M[t]^{(n)}(x) \right|^2\,dx\biggr]^{1/2}.
\end{align*}
Thus we actually conclude that for any fixed compact $K \subset [0,\infty)$
\[
\sup_{t \geq 0} \int_K \left| \M[t]^{(n)}(x) \right|\,dx \to 0
\]
as $n \to \infty.$  All said, we have proven:
\begin{lemma}
\label{lem:CS_noise}
If $I_t \to 0$ almost surely, then $\M^{(n)} \Dto 0$ almost surely.
\end{lemma}
We now turn to estimating $I_t,$ which by a change of variables we can represent as
\[
I_{t} = e^{-2t}\int_0^\infty\frac{1}{x^3}( \Mtilde[t](x))^2\,dx.
\]
Hence, we set $J_{t}$ to be
\[
J_t = \int_0^\infty\frac{1}{x^3}( \Mtilde[t](x))^2\,dx,
\]
and note that
as $\Mtilde(x)$ is a martingale for every $x$, this is a submartingale.  Thus, by virtue of Doob's maximal inequality, to control its supremum in $t$, it is enough to control its expectation. Taking expectations, we have
\begin{align}
\label{eq:Jtk}
\Exp J_{t}
= \int_0^\infty \frac{1}{x^3} \Exp \left\langle \Mtilde(x) \right\rangle_t\,dx,
\end{align}
with $\left\langle \M(x) \right\rangle_t$ the predictable quadratic variation.

As for the predictable quadratic variation, we have the following bound.
\begin{lemma}
\label{lem:qv}
For any $x \geq 0$ and any $t \geq 0,$
\[
\left\langle \Mtilde{(x)} \right\rangle_t
\leq\frac{2x^3}{3}
\int_0^t e^{s}\int_x^\infty \frac{1}{z} d\Psi(\Atilde[s](z))\,ds.
\]
\end{lemma}
\begin{proof}[Proof of Lemma~\ref{lem:qv}]
At a point $(s,u,v) \in \Pi,$
the quadratic variation of $M_s(x)$ increases by at most ${B(s,u,v,x)}^2.$ As the process is pure jump, we may write that
\begin{align*}
\left\langle \Mtilde(x) \right\rangle_{t}
\leq\int\limits_{0}^t f(s,x) e^s\,ds
\end{align*}
with $f(s,x)$ given by ${B(s,u,v,x)}^2$ conditional on $\filt_{s-}$ and on the event that there is a jump at $s,$ i.e.
\[
f(s,x) =
\iint
{B(s,u,v,x)}^2\,d\Psi(u)dv.
\]
Doing the $v$ integral and applying the convexity of $x^2,$ we may bound this by
\begin{align*}
f(s,x) =
\iint
{B(s,u,v,x)}^2\,d\Psi(u)dv
\hspace{-1.5in}&\hspace{1.5in} \\
&\leq
\int\limits_{0}^1
\ell_s(u)^2
\one[\ell_s(u) > x]
\int\limits_{0}^1
2v^2 \one[\ell_s(u) v \leq x]\,dv\,d\Psi(u) \\
&=
\int\limits_{0}^1
\ell_s(u)^2
\one[\ell_s(u) > x]
\frac{2x^3}{3 \ell_s(u)^3}\,d\Psi(u) \\
&=
\frac{2x^3}{3} \int_x^\infty \frac{1}{z} d\Psi(\A[s-](z)).
\end{align*}
\end{proof}

With the quadratic variation estimate, the desired result about $I_{t}$ follows immediately.
\begin{lemma}
\label{lem:M_decay}
With probability $1,$ we have that
\[
\limsup_{t \to \infty} I_{t} = 0.
\]
In particular, with probability $1,$ we have $\M^{(n)} \Dto 0$ as $n\to\infty$.
\end{lemma}
\begin{proof}
From~\eqref{eq:Jtk} and Lemma~\ref{lem:qv}, we have that
\[
\Exp J_t \leq \Exp\int_0^\infty \frac{1}{x^3}
\frac{2x^3}{3}
\int_0^t e^{s}\,ds
\int_x^\infty \frac{1}{z} d\Psi(\Atilde[s](z))\,dx.
\]
By applying the Fubini-Tonelli theorem, we have that
\begin{align*}
\int_0^\infty
\int_0^t e^{s}\,ds
\int_x^\infty \frac{1}{z} d\Psi(\Atilde[s](z))\,dx
&=
\int_0^t e^{s}
\int_0^\infty \frac{1}{z}
\int_0^z dx\,d\Psi(\Atilde[s](z))\,ds \\
&=
\int_0^t e^{s}
\int_0^\infty d\Psi(\Atilde[s](z))\,ds \\
&=
\int_0^t e^{s}\,ds.
\end{align*}
By Doob's maximal inequality, we have
\[
\Pr\left[
\sup_{0 \leq s \leq t} J_s > e^{3t/2}
\right] \leq e^{-t/2}.
\]
Taking $t$ to run over the natural numbers, we may apply Borel Cantelli to conclude there is a random $T < \infty$ so that $J_t \leq Ce^{3t/2}$ for all $t > T.$  Hence, as $I_t = e^{-2t}J_t,$ we get that $I_t \to 0$ with probability $1.$
\end{proof}

\section{Proof of Theorem~\ref{thm:main0} and Lemma~\ref{lem:candy_portmanteau}}
\label{sec:proof}

\begin{proof}[Proof of Theorem~\ref{thm:main0}]
Recall that $N_t$ is the number of points of $\Pi$ on $[0,t].$ As mentioned at the beginning of Section~\ref{sec:definitions}, we can realize $\diFtilde[n]$ in terms of $\Atilde[t]$ in such a way that
\(\Atilde[t] = \diFtilde[N_t]\) holds for all $t \geq 0.$ Recall that $\A[t](x) = \Atilde[t](e^{-t}x)$ and $\diF[N_t](x) = \diFtilde[N_t](x / (N_t+n_0)).$  By Theorem~\ref{thm:main}  and the fact that $e^{-t}(N_t+n_0) \to 1$ almost surely as $t\to \infty$, $D_n$ now converges almost surely to $F^\Psi$. Furthermore, $\candy[D_n] = \candy[F^\Psi] = 1$ for all $n$. By Lemma~\ref{lem:candy_portmanteau}, this implies that $\mu_n$ (weakly) converges almost surely to $\mu^\Psi$.
\end{proof}

\begin{proof}[Proof of Lemma~\ref{lem:candy_portmanteau}]
	\noindent {$(i) \implies (ii):$}
As $\dblah(F_n,F) \to 0,$ we have $F_n \to F$ in measure.
Since $x^{-2}|F_n(x) - F(x)| \leq x^{-2}(F_n(x) + F(x))$ for all $x > 0$ and since by assumption
\[
	\int_0^\infty x^{-2}(F_n(x) + F(x))\,dx 
	\to
	\int_0^\infty 2x^{-2}F(x)\,dx, 
\]
as $n \to \infty,$
then by sequential dominated convergence,
\[
	\dcandy(F_n,F) = \int_0^\infty x^{-2}|F_n(x) - F(x)|\,dx \to 0
\]
as $n \to\infty.$

\vspace{10pt}
\noindent $(ii) \implies (i):$
This implication follows immediately from \eqref{eq:dblah_domination}, which we recall for convenience.
For any $F,G \in \DSpace_1,$ we have that for any $K \in \N,$
\begin{equation*}
\dblah(F,G) \leq 2^{-K} + \int_0^K |F(x) - G(x)|\,dx \leq 2^{-K} + K^2 \dcandy(F,G),
\end{equation*}
from which the desired implication follows.


\vspace{10pt}
\noindent $(i) \implies (iii):$

Let $G \in \DSpace_1.$  By taking contrapositives, we have that the finiteness of $\int_0^\infty x^{-2}G(x)\,dx$ implies that $\liminf_{\epsilon \to 0} \epsilon^{-1}G(\epsilon)=0.$  Hence, there is a decreasing sequence $(\epsilon_k)_{k=1}^\infty$ along which $\lim_{k \to\infty} \epsilon_k^{-1}G(\epsilon_k) = 0.$  Applying integration by parts, we get that for any $y > 0$
\[
	\int_{\epsilon_k}^y x^{-1}\,dG(x) = y^{-1}G(y) - \epsilon_k^{-1}G(\epsilon_k) + \int_{\epsilon_k}^y x^{-2}G(x)\,dx.
\]
Taking $k \to \infty,$ we have by monotone convergence that
\[
	\int_{0}^y x^{-1}\,dG(x) = y^{-1}G(y) + \int_{0}^y x^{-2}G(x)\,dx.
\]
Hence we may apply this representation to $F_n$ to get that for all $y > 0,$
\[
	\int_{0}^y x^{-1}\,dF_n(x) = y^{-1}F_n(y) + \int_{0}^y x^{-2}F_n(x)\,dx.
\]
From $(i)$ and Lemma~\ref{lem:L1loc_pointwise} we have that $F_n \to F$ almost everywhere, and hence for almost every $y > 0,$ $y^{-1}F_n(y) \to y^{-1}F(y).$  From $(ii),$ we have that the integral term converges to the same with $F(x)$ in place of $F_n(x)$ for all $y > 0.$  Hence we have shown that for almost all $y > 0,$
\[
	\int_{0}^y x^{-1}\,dF_n(x)
	\to
	\int_{0}^y x^{-1}\,dF(x),
\]
which implies the weak convergence of the measures.

\vspace{10pt}
\noindent $(iii) \implies (i):$
For every compactly supported continuous function $\phi : [0,\infty) \to \R,$ we have that $x \mapsto x \cdot \phi(x)$ is again compactly supported and continuous.  Hence 
	\[
		\int\limits_0^\infty \phi(x)dF_n(x) 
		=
		\int\limits_0^\infty x\phi(x)x^{-1}dF_n(x) 
		\to
		\int\limits_0^\infty x\phi(x)x^{-1}dF(x) 
		=
		\int\limits_0^\infty \phi(x)dF(x) 
	\]
	as $n\to \infty.$  By a standard argument, this implies the almost everywhere convergence of $F_n$ to $F$ and hence that $\dblah(F_n,F) \to 0$ by dominated convergence. 
	From weak convergence, we have that $\int_0^\infty x^{-1}\,dF_n(x) \to \int_0^\infty x^{-1}\,dF(x).$   Hence by Lemma \ref{lem:candy_fubini}, we conclude that $\candy[F_n] \to \candy[F],$
	which completes the proof.   
\end{proof}

\section{Properties of limiting profile}
\label{sec:limit_profile}

In this section, we study properties of the distribution function $F = F^\Psi$ from Lemma~\ref{lem:unique_F*}, i.e.\ the distribution function of the size-biased empirical measure of interval lengths in the limit as the number of intervals goes to infinity.
We have the following lemma:
\begin{lemma}
\label{lem:ode}
Assume $\Psi$ is absolutely continuous with derivative $\psi$ and $\FPSpace \ne \emptyset$. Let $F=F^\Psi$ be the limiting distribution function from Lemma~\ref{lem:unique_F*}. Then $F\in C^1([0,\infty),[0,1]),$ its derivative $F'$ is absolutely continuous and $F$ satisfies
\begin{equation}
xF''(x)-F'(x)+xF'(x)\psi(F(x)) = 0
\label{eq:F_ode}
\end{equation}
for almost every $x\ge0.$ If $\psi$ is continuous on $[0,\infty)$, then $F\in C^2([0,\infty),[0,1])$ and \eqref{eq:F_ode} holds for every $x\ge0$.
\end{lemma}
\begin{proof}
  By Lemma~\ref{lem:unique_F*}, and the assumption on $\Psi$, $F\in C^1([0,\infty),[0,1])$  and satisfies
  \begin{equation}
  \label{eq:night}
    F'(x) = x\int_x^\infty \frac 1 z\,d\Psi(F(z))
  \end{equation}
  for every $x\ge0.$ Since $F$ is continuously differentiable and monotone, the function $\Psi\circ F$ is then absolutely continuous as well and $(\Psi \circ F)'(z) = \psi(F(z))F'(z)$ for almost every $z\ge 0$ (this is easy to see but is also contained in \cite[Exercise~3.51]{Leoni}). It follows that the function $x\mapsto F'(x)/x$ is absolutely continuous and hence $F'$ is absolutely continuous as well.  Dividing both sides of \eqref{eq:night} by $x$ and differentiating, we get that~\eqref{eq:F_ode} holds for almost every $x\ge0.$
  
  If $\psi$ is continuous, then $\Psi\circ F$ is $C^1$ with $(\Psi \circ F)'(z) = \psi(F(z))F'(z)$ for every $z\ge0$. As above, it then easily follows from \eqref{eq:night} that $F\in C^2$ and that \eqref{eq:F_ode} holds for every $x\ge0$.
\end{proof}
In what follows, we study the right tail of the distribution function $F=F^\Psi$ from Lemma~\ref{lem:unique_F*} (it is easily seen that the assumptions in the following statements imply (C) and (D), such that the assumptions of Lemma~\ref{lem:unique_F*} are verified by the virtue of Theorem~\ref{thm:main}). We first study its right tail.
\begin{proposition}
\label{prop:F_tail_max}
Assume that $\Psi$ is absolutely continuous with derivative $\psi$ satisfying $\lim_{u\to 1} \psi(u) = \psi(1) > 0$.
\begin{enumerate}
	\item For every $a < \psi(1)$, we have $F'(x) \le e^{-ax}$ for large $x$.
	\item If furthermore there exists $\beta > 1/\psi(1)$, such that $|\psi(1)-\psi(1-u)| \le |\log u|^{-\beta}$ for small enough $u$, then as $x\to\infty$,
\[
F'(x) \sim Cx\exp(-\psi(1)x),
\]
for some $C>0$.
\end{enumerate}
\end{proposition}
\begin{corollary}
In the max-$k$ process (i.e.\ $\Psi(u)=u^k$), there exists $C>0$, such that $F'(x) \sim Cx\exp(-kx)$ as $x\to\infty$.
\end{corollary}

\begin{proof}[Proof of Proposition~\ref{prop:F_tail_max}]
  Lemma~\ref{lem:ode} implies that
\[
F''(x) = \left(\frac 1 x -\psi(F(x))\right) F'(x).
\]
Rearranging, this implies that
\[
  \frac{d}{dx}\left(F'(x)\exp\left(-\int_1^x \frac 1 y - \psi(F(y))\,dy\right)\right) = 0.
\]
Integrating this equation gives
\begin{align}
  \nonumber
F'(x) &= F'(1)\exp\left(\int_1^x \frac 1 y - \psi(F(y))\,dy\right) \\
  \label{eq:F_prime}
&= F'(1)x\exp\left(-\int_1^x \psi(F(y))\,dy\right).
\end{align}
Now assume that $\lim_{u\to 1} \psi(u) = \psi(1) > 0$. Since $F(x)\to 1$ as $x\to\infty$, the first statement follows directly from \eqref{eq:F_prime}. Now assume that there exists $\beta > 1/\psi(1)$, such that $|\psi(1)-\psi(1-u)| \le |\log u|^{-\beta}$ for small enough $u$. Let $a\in(\beta^{-1},\psi(1)^{-1})$. Then by the first statement, we have $1-F(x) \le e^{-ax}$ for large $x$, which implies that
\[
|\psi(1)-\psi(F(x))| \le \frac 1 {x^{a\beta}},\quad\text{for large $x$.}
\]
In particular, the integral $\int_0^x (\psi(1)-\psi(F(y)))\,dy$ converges to a limit as $x\to\infty$. Together with \eqref{eq:F_prime}, this now implies that
\[
F'(x) = F'(1)x \exp\left(-\psi(1) x + \int_0^x (\psi(1) - \psi(F(y)))\,dy\right) \sim C'x\exp(-\psi(1) x),
\]
as $x\to\infty$, for some $C'>0$. This finishes the proof of the proposition.
\end{proof}

In contrast to the case treated in Proposition~\ref{prop:F_tail_max}, the case $\lim_{u\to 1} \psi(u) = \psi(1) = 0$ is more delicate. Here we are only able to give a satisfying answer for the min-$k$ process, i.e.\ $\psi(u) = k(1-u)^{k-1}$ for $k > 1$ (not necessarily integer). In this case, we are able to transform equation \eqref{eq:F_ode} into an \emph{autonomous} differential equation by setting $F(x) = 1-G(\log x^{1/(k-1)})/x^{1/(k-1)}$. This equation can then be studied by standard phase plane analysis, yielding the following result:
\begin{proposition}
\label{prop:F_tail_min}
Assume $\Psi(u) = 1-(1-u)^k$  for some (real) $k > 1$. Then, as $x\to\infty$,
\[
1-F(x) \sim \frac{c_k}{x^{1/(k-1)}}\text{ and }F'(x) \sim \frac{c_k}{(k-1)x^{1+1/(k-1)}},
\]
where $c_k = ((2k-1)/k(k-1))^{1/(k-1)}$.
\end{proposition}
\begin{proof}
 Since $\Psi$ is continuously differentiable, Lemma~\ref{lem:ode} implies that $F\in C^2$ and \eqref{eq:F_ode} holds for every $x\ge0$.  Set $\overline F(x) = 1-F(x)$, so that this differential equation becomes
\[
x\overline F''(x) -\overline F'(x) + xk\overline F(x)^{k-1}\overline F'(x) = 0.
\]
The substitution $\overline F(x) = G(\log x^{1/(k-1)})/x^{1/(k-1)}$ or $G(t) = e^t\overline F(e^{(k-1)t})$ then yields the autonomous differential equation
\begin{equation}
G'' - G' - (2k-1 - k(k-1)G^{k-1})(G'-G) = 0.
\label{eq:G}
\end{equation}
We can now study \eqref{eq:G} by usual phase plane analysis (see e.g.\ \cite{Arnold1992}). For this, we consider the two-dimensional ODE
\begin{equation}
\label{eq:2d_ode}
\begin{pmatrix}
G\\
G'
\end{pmatrix}'
=
\begin{pmatrix}
G'\\
G''
\end{pmatrix}
=
\begin{pmatrix}
G'\\
G' + (2k-1 - k(k-1)G^{k-1})(G'-G)
\end{pmatrix}.
\end{equation}
We recall that a critical point of the ODE \eqref{eq:2d_ode} is a point $(a,b)\in \R^2$ such that the right-hand side of $\eqref{eq:2d_ode}$, with $G=a$ and $G'=b$, vanishes. Obviously, $(a,b)$ is a critical point of \eqref{eq:2d_ode} if and only if $b=0$ and 
\[
(2k-1-k(k-1)a^{k-1})a = 0,
\]
i.e.\ if and only if $a\in\{0,c_k\}$. Since the right-hand side of \eqref{eq:2d_ode} is locally Lipschitz-continuous in $(G,G')$, general theory then tells us that non-critical points can only be reached in finite time and critical points can only be reached in infinite time unless $G$ is constant (which it isn't, since $G(-\infty) = 0$ and $\overline F$ is not constant equal to 0). We will use these facts below without further mention.

We now study the possible orbits of solutions to \eqref{eq:2d_ode} and identify the one that corresponds to our particular solution. We first note that since $\overline F(0) = 1$, we have $G(-\infty) = 0$. Furthermore, $G\ge 0$. We will now rule out several orbits through a chain of arguments.
\begin{enumerate}
  \item \emph{$G(t)$ cannot go to $\infty$ as $t\to\infty$.}\\
  For, suppose that $\lim_{t\to\infty} G(t) = \infty$. Define $H(t) = G(t)-G'(t)$. Equation \eqref{eq:G} then implies that for large $t$, $\tfrac 1 2 \tfrac d {dt} H(t)^2 = H(t)H'(t) < -H(t)^2.$  By Gronwall's inequality, this implies that $|H(t)| \to 0.$  As $G(t) \to \infty$ this implies that $G'(t) > G(t) - 1$ for large $t$. Gronwall's inequality then shows that $G(t) > Ce^t$ for some $C>0$ and for large $t$. But this contradicts the fact that $G(t) = e^t\overline F(e^{(k-1)t}) = o(e^t)$, which follows from the fact that $\overline F(x)\to 0$ as $x\to\infty$. This contradiction shows that $G(t)$ cannot go to $\infty$ as $t\to\infty$.
  \item \emph{$G'(t) \ne 0$ for all $t\in\R$.}\\
  For, suppose there exists $t_0\in\R$, such that $G'(t_0) = 0$. Then $G(t_0) \ne \{0,c_k\}$ because critical points can only be reached in infinite time. We will distinguish two cases.
  
  {Case $G(t_0)\in(0,c_k)$:} By \eqref{eq:G}, we have $G''(t_0) < 0$, so that there exists $\ep>0$ and $t_1>t_0$, such that $(G(t_1),G'(t_1)) \in (0,c_k)\times (-\infty,-\ep)$. For $(G,G')$ in this domain, we have $G'' \le -\ep$ by \eqref{eq:G}, whence the orbit cannot exit this domain at the $G' = -\ep$ border, neither at the $G=c_k$ border since $G$ is decreasing in this domain. Furthermore, the orbit cannot stay forever inside the domain, because then $G'(t) = G'(t_1) + \int_{t_1}^t G''(s)\,ds\to -\infty$ and therefore $G(t)\to-\infty$ as $t\to\infty$. Hence, the orbit has to exit the domain at the $G=0$ border. But since then $G' < -\ep$, we would have $G(t) < 0$ for some $t\in\R$, which is in contraction with $G\ge0$. 
  
  {Case $G(t_0) > c_k$:} In this case, \eqref{eq:G} gives $G''(t_0) > 0$, whence $t_0$ is a (strict) local minimum of $G$. Since $G(-\infty) = 0$, by continuity there must then exist a $t_1 < t_0$, such that $G(t_1) = \max\{G(t):t\le t_0\} > G(t_0)$. At $t_1$, we then have $G'(t_1) = 0$ and $G''(t_1) \le 0$. But by \eqref{eq:G}, we again have $G''(t_1) > 0$, which is a contradiction.
  \item \emph{$G'(t) > 0$ for all $t\in\R$.}\\
  Since $G'(t) \ne 0$ for all $t\in\R$, we either have $G'>0$ or $G'<0$. Since $G(-\infty) = 0$, the latter would imply that $G(t) < 0$ for some (indeed, all) $t\in\R$, which is in contradiction with $G\ge 0$. Hence, $G'(t) > 0$ for all $t\in\R$.
  \item \emph{$(G(\infty),G'(\infty)) = (c_k,0)$.}\\
  The previous points imply that $G$ is non-decreasing, non-zero and bounded. In particular, $(G(t),G'(t))$ converges to a limit $(c,0)$ as $t\to\infty$,  with $c>0$. This limit has to be a critical point. Since the only such critical point is $(c_k,0)$, we have $c=c_k$.
\end{enumerate}
The preceding results now give as $x\to\infty$,
\begin{align*}
\overline F(x) &\sim \frac {c_k} {x^{1/(k-1)}},\quad\text{ and }\\
\overline F'(x) &= \frac{-G(\log x^{1/(k-1)})+G'(\log x^{1/(k-1)})}{(k-1)x^{1+1/(k-1)}}\sim -\frac {c_k} {(k-1)x^{1+1/(k-1)}}.
\end{align*}
This finishes the proof of the lemma.
\end{proof}
\begin{remark}
The critical point $(c_k,0)$ of the ODE \eqref{eq:2d_ode} is in fact a saddle point. This can be used to yield another proof of the uniqueness of the solution to \eqref{eq:F_ode} with $\psi(u) = k(1-u)^{k-1}$ and the boundary conditions $F(0)=0$, $F(\infty) = 1$, $F\ge 0$.
\end{remark}

Lastly, we study the asymptotics when the measure $d\Psi$ converges weakly to $\delta_1$ (which corresponds to the \emph{Kakutani process} as mentioned in the introduction). Formally, the function $F$ satisfies in this case the equation
\[
 xF'' - F' + xF'(x-)\delta_1(F(x)) = 0,
\]
which implies that $F$ is of the form $F(x) = Cx^2 \wedge 1$ for some $C>0$. Since $F'(x)/x$ is the density of the interval distribution, we have with  $x_0=1/\sqrt C$,
\[
 1 = \int_0^\infty F'(x)/x\,dx = 2Cx_0 = 2\sqrt C \Rightarrow C = 1/4.
\]
The following proposition makes this argument rigorous:

\begin{proposition}
\label{prop:limit_kakutani}
Let $(\Psi_n)_{n \geq 0}$ be a sequence of distribution functions of measures on $(0,1]$ with $\Psi_n(x) \to 0$ for all $x \in (0,1).$  Assume that for all these $\Psi_n,$ there are distribution functions $F_n$ satisfying
\begin{align}
\label{eq:F_integral_eq}
F_n(x) =
\int_0^x y
\int_y^\infty
\frac{1}{z}
d\Psi_n(F_n(z))\,dy
\end{align}
for all $x \geq 0.$  Then $F_n(x)\to x^2/4\wedge 1$ pointwise as $n\to\infty$.
\end{proposition}
\begin{proof}
It follows immediately from the integral equation satisfied by $F_n$ that it is absolutely continuous and thus satisfies~\eqref{eq:F_ode_integral} for almost every $x.$

We begin by showing that the $\{F_n\}_{n \geq 0}$ are tight. Dividing~\eqref{eq:F_ode_integral} by $x$ and integrating, we get that
\begin{align}
\nonumber
\int_x^\infty \frac{dF_n(y)}{y}
&=\int_x^\infty \int_y^\infty \frac{d\Psi_n(F_n(z))}{z}\,dy \\
\nonumber
&=\int_x^\infty \frac{z-x}{z}{d\Psi_n(F_n(z))} \\
\nonumber
&\geq\frac{1}{2}\int_{2x}^\infty{d\Psi_n(F_n(z))} \\
\label{eq:pascal}
&=\frac{1}{2}(1-\Psi_n(F_n(2x))).
\end{align}
On the other hand, we get that
\begin{align}
\label{eq:maillard}
\int_x^\infty \frac{dF_n(y)}{y} \leq \frac{1}{x}\int_x^\infty {dF_n(y)} \le \frac 1 x.
\end{align}
From the convergence of $\Psi_n \to 0,$ we have that for any $\delta > 0$ and any $\epsilon > 0,$ there is an $n_0$ sufficiently large so that for $n \geq n_0,$ $\Psi_n(u) \leq \epsilon$ for $u \leq 1-\delta.$  Thus, combining~\eqref{eq:pascal} and~\eqref{eq:maillard} we get that
\begin{align*}
1-\frac{2}{x}
&\leq \Psi_n(F_n(2x)) \\
&\leq \epsilon \one[F_n(2x) \leq 1-\delta] + \one[F_n(2x) > 1-\delta] \\
&\leq \epsilon + \one[F_n(2x) > 1-\delta].
\end{align*}
Setting $x=4$ in the above equation and assuming $\epsilon < 1/2$, we have $F_n(8) > 1-\delta$ for all $n\ge n_0$. This implies tightness of the sequence $(F_n)_{n \geq 0}.$

Integrating~\eqref{eq:F_integral_eq} by parts, we have that
\[
F_n(x) = -\int_0^x \Psi_n(F_n(z))\,dz + \int_0^x z \int_z^\infty \frac{1}{y^2} \Psi(F_n(y))\,dydz.
\]
By passing to a convergent subsequence, we may assume that there is a nondegenerate distribution function $F_*$ so that $F_n \to F_*$ at every point of continuity of $F_*.$  We then get that $\Psi_n(F_n)$ converges almost everywhere to $\one[F_*(x) \geq 1].$  By dominated convergence, we can pass to the limit in the previous equation to get
\[
F_*(x) =
-\int_0^x\one[F_*(z) \geq 1]\,dz
+\int_0^x z \int_z^\infty \frac{1}{y^2}\one[F_*(y) \geq 1] \,dydz.
\]
Let $x_0 = \sup\{x: F_*(x) < 1\}.$  Note that if $x_0 = \infty,$ then both integrals are identically $0,$ implying $F_* \equiv 0$ and contradicting the tightness of $F_n.$  For any $x \leq x_0,$ we get that
\begin{align*}
F_*(x) = \int_0^x z \int_z^\infty \frac{1}{y^2}\one[F_*(y) \geq 1] \,dydz = \int_0^x \frac{z}{x_0} dz = \frac{x^2}{2x_0} .
\end{align*}
This forces $x_0 = 2,$ and hence $F_*\equiv x^2/4\wedge 1.$  As this holds for every subsequential limit of $F_n,$ we have completed the proof.
\end{proof}

\bibliographystyle{alpha}


\bibliography{IntervalDivision,Achlioptas}

\end{document}